\documentclass[12pt, a4paper, english]{article}

\usepackage[utf8]{inputenc}
\usepackage[T1]{fontenc}
\usepackage[nomath]{kpfonts}

\DeclareUnicodeCharacter{1ECD}{\d{o}}

\usepackage{amssymb}
\usepackage{mathtools}
	\DeclareMathOperator{\curv}{Curv}
	\DeclareMathOperator{\supp}{supp}

	\DeclareMathOperator{\wfs}{\textup{WF}}

\usepackage{caption}

\usepackage{nicematrix}

\usepackage{hyperref}
	\hypersetup{
		breaklinks=true,   
		pdfusetitle=true,  
	}

\usepackage[backend=biber,style=trad-abbrv]{biblatex}

\usepackage{authblk}

\usepackage{geometry}

\usepackage{tikz}
	\usetikzlibrary{cd}

\newcommand{\fbarg}{\mathbf{F}_N}

\newcommand{\fbi}[1][\hbar]{\mathcal{B}_{#1}}
	
\newcommand{\Cm}{\mathbb{C}}
\newcommand{\Nm}{\mathbb{N}}

\newcommand{\Rm}{\mathbb{R}}

\newcommand{\Tm}{\mathbb{T}}

\newcommand{\lpar}{\left(}
\newcommand{\rpar}{\right)}
\newcommand{\lacc}{\left\{}
\newcommand{\racc}{\right\}}
\newcommand{\lver}{\left|}
\newcommand{\rver}{\right|}
\newcommand{\lnor}{\left\|}
\newcommand{\rnor}{\right\|}

\usepackage{babel}[english]
	\usepackage{csquotes}

\usepackage{amsthm}
	\newtheorem{theorem}{Theorem}[section]
	\newtheorem*{theorem*}{Theorem}
	\newtheorem{corollary}{Corollary}[theorem]
	\newtheorem{lemma}[theorem]{Lemma}
	\theoremstyle{definition}
	\newtheorem{definition}{Definition}[section]
	\newtheorem{example}{Example}[section]
	\newtheorem{proposition}{Proposition}[section]
	\newtheorem{corollaryp}{Corollary}[proposition]

\title{Semiclassical concentration estimates for Berezin-Toeplitz quasimodes for regular energies}

\author{Nathan Réguer\footnote{Univ Rennes, CNRS, IRMAR - UMR 6625, F-35000 Rennes, France.\\nathan.reguer@univ-rennes.fr\hfill MSC codes: 81Q10, 81Q20, 35Q40}}

\date{}

\begin{document}

\maketitle

\begin{abstract}
The purpose of this article is to prove sharp $L^p$ bounds for quasimodes of Berezin-Toeplitz operators. We consider examples with explicit computations and a general situation on compact spaces and $\Cm^n$. In both cases the eigenvalue is a regular value of the operator symbol. We then use the link between pseudodifferential and Berezin-Toeplitz operators to obtain an $L^p$ bound of the FBI transform of quasimodes of pseudodifferential operators.

\end{abstract}

\tableofcontents

\section{Introduction}

Let $a\in C^{\infty}(\Rm^n)$ be bounded with all its derivatives, and $Op_{\hbar}^W(a)$ its semiclassical Weyl quantisation given by, for $u_{\hbar}$ in the Schwartz space,

\[
Op_{\hbar}^W(a)u_{\hbar} = (2\pi\hbar)^{-n} \int_{\Rm^n\times\Rm^n}e^{\frac{i(x-y)\cdot\xi}{\hbar}}a\lpar\frac{x+y}{2},\xi\rpar u(y)dyd\xi.
\]
If $u_{\hbar}\in L^2(\Rm^n)$ is a quasimode for $Op_{\hbar}^W(a)$, which means there exists $\lambda \in\Rm$ and a sequence $(\lambda_{\hbar})$ that converges to $\lambda$ with

\[
Op_{\hbar}^W(a) u_{\hbar} = \lambda_{\hbar} u_{\hbar} + O(\hbar^{\infty})\lnor u_{\hbar} \rnor_{L^2(\Rm^n)}
\]
as $\hbar \rightarrow 0$, then we know that the function $u_{\hbar}$ concentrates in  \textit{phase space} near the classical energy level $a^{-1}(\{\lambda\})$. More precisely, for any $b \in C^{\infty}_c(\Rm^n)$ whose support does not intersect $a^{-1}(\{\lambda\})$,

\[
\lnor Op_{\hbar}^W(b)u_{\hbar} \rnor_{L^2(\Rm^n)} = O(\hbar^{\infty}) \lnor u_{\hbar} \rnor_{L^2(\Rm^n)}.
\]
This concentration is a consequence of the symbolic calculus of pseudodifferential operators. A proof can be found in \cite{zwor12} for Schrödinger type operators, and it also works for $a$ as above. The elementary properties of pseudodifferential operators can be found in \cite{dima99}.

An analogous bound can be proved for Berezin-Toeplitz operators, which is a framework where the wave functions are defined on the phase space, here $\Rm^{2n}$, seen as $\Cm^n$. We will recall the definition of these operators later. The result states that if $u$ is a quasimode of a Berezin-Toeplitz with symbol $f:\Cm^n\rightarrow \Rm$ and with eigenvalue $\lambda$, then for any open set $W\subset \Cm^n$ at positive distance to $f^{-1}(\lacc \lambda\racc)$

\[
\lnor u \rnor_{L^2(W)} = O(\hbar^{\infty}) \lnor u \rnor_{L^2}.
\]
This estimate can be proved using the symbolic calculus of Toeplitz operators, as developed in \cite{bord94} and \cite{char03a}. A proof can be found in \cite{dele16} Proposition 3.1.

A possible way to compare the speed at which quasimodes concentrate is to bound their $L^p$ norms with respect to their $L^2$ norms for $2<p\le\infty$. This topic is actively searched since the 1980’s. First, Sogge \cite{sogg88} proved a result for self-adjoint second-order elliptic operators on smooth connected compact manifolds of dimension $n\ge 2$. If $P$ is such an operator, and if we denote by $\chi_k$ the projector on the space spanned by the eigenfunctions whose eigenvalues $\lambda$ satisfy $k-1 \le \sqrt{|\lambda|}<k$, then for all $u\in L^2(M)$,

\begin{align}
\label{eq_maj_Sogge}
\lnor \chi_k u \rnor_{L^{2\frac{n+1}{n-1}}(M)} & \le C k^{\frac{n-1}{2(n+1)}} \lnor u \rnor_{L^2(M)},\\
\lnor \chi_k u \rnor_{L^{\infty}(M)} & \le C k^{\frac{n-1}{2}} \lnor u \rnor_{L^2(M)}. \nonumber
\end{align}

Then, $L^p$-$L^q$ interpolation gives a bound for any $L^p$ norm. Namely, we have

\begin{equation}
\label{eq_maj_inter}
\lnor \chi_k u \rnor_{L^p(M)} \le C k^{\rho\lpar\frac{1}{p}\rpar} \lnor u \rnor_{L^2(M)},
\end{equation}

with $\rho$ given by the red curve in Figure \ref{fig_Lp}.

\begin{figure}[h]
\centering
\begin{tikzpicture}
\draw (-4,-0.8)--(-4,-1.2) node[black, below]{$0$};
\draw (-3.8,-1) -- (-4.2,-1) node[black, left]{$0$};
\draw[black,very thick,->] (-4,-1) -- (2,-1) node[black, right]{$\frac{1}{p}$};
\draw[black,very thick,->] (-4,-1) -- (-4,5) node[black, left]{$\rho$};
\draw (-3.8,4) -- (-4.2,4) node[black, left]{$\frac{n-2}{2}$};
\draw[red] (-4,4) -- (-2,1);
\draw (-2,-0.8) -- (-2,-1.2) node[black, below]{$\frac{n-1}{2(n+1)}$};
\draw[dashed] (-2,1) -- (-2,-1);
\draw (-3.8,1) -- (-4.2,1) node[black, left]{$\frac{n-1}{2(n+1)}$};
\draw[dashed] (-2,1) -- (-4,1);
\draw[red] (-2,1) -- (1,-1);
\draw (1,-0.8)--(1,-1.2) node[black, below]{$\frac{1}{2}$};
\draw[dashed,blue] (-4,4) -- (1,-1);

\matrix [draw,below left] at (current bounding box.north east) {
	\draw[red] (-0.5,0) -- (0,0) node[right,color=black]{Sharp estimate}; \\
	\draw[dashed,blue] (-0.5,0) -- (0,0) node[right,color=black]{$L^2$-$L^{\infty}$ interpolation};\\
};

\end{tikzpicture}
\caption{The exponent $\rho$ as a function of $\frac{1}{p}$.}
\label{fig_Lp}
\end{figure}
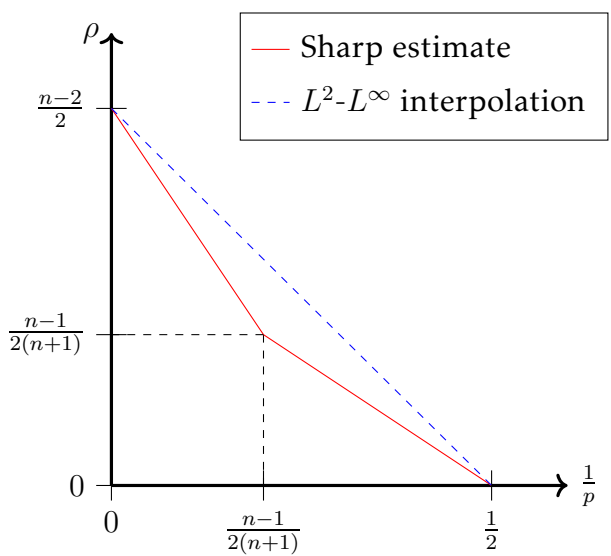

In Figure \ref{fig_Lp}, the dashed blue line corresponds to the exponent given by $L^p$-$L^q$ interpolation if we consider only the first line in equation \eqref{eq_maj_Sogge}, but it is less precise than the red line. Actually, equation \eqref{eq_maj_inter} cannot be improved, as Sogge identified that the bound becomes an asymptotic equivalence for specific spherical harmonics. We say that the estimate is sharp. In general, concentration results give estimates of the form of equation \eqref{eq_maj_inter} with affine by parts exponents. From any $L^{\infty}$ estimate, it is easy to get an affine exponent by $L^p$-$L^q$ interpolation, but the sharp exponents are more difficult to obtain as they may depend on caustics. For instance, there are two families of spherical harmonics for the Laplace-Beltrami operator that saturate equation \eqref{eq_maj_inter}. They saturate the left and right affine parts of the red line in Figure \ref{fig_Lp}, respectively.

Since then, a lot of results generalised these estimates in many cases with different caustics, and thus different exponents. Koch, Tataru \cite{koch05} proved similar bounds for the Hermite operator $-\Delta+x^2$ in $\Rm^n$ with $n\ge 3$, and then for more general operators, motivated by strong unique continuation for parabolic equations. Koch, Tataru and Zworski \cite{koch07} obtained microlocal quasimdodes $L^p$ bounds for a larger class of operators. In these examples, the exponent is always an affine by parts function, and we expect a different saturating family of eigenfunction on each affine part.

However, to the author's knowledge, such estimates are yet to be proved for Berezin-Toeplitz operators. The closest result is in \cite{chan22}, and \cite{paol24} for operators on boundaries of Grauert tubes of Riemanian manifolds, since they also consider wave functions defined on the phase space.  Such frameworks have the advantage that there is no caustic since the wave functions are defined on the phase space,  and the estimates are then simpler.

This article focuses on sharp $L^p$ norm estimates of quasimodes of Berezin-Toeplitz operators in the semiclassical regime. We recall the construction of these operators on $\Cm^d$, and their link with pseudodifferential operators. First, we recall the definition of the wave front set, which is the set where a function concentrates, in the sense we saw earlier. For $u_{\hbar}\in L^2(\Rm^n)$, the wave front set $\wfs_{\hbar}(u_{\hbar})$ of $u$ is the complement of points $(x,\xi)$ for which there exists $a\in C^{\infty}_c(\Rm^{2n})$, with $a(x,\xi)>0$, and

\[
\lnor Op_{\hbar}^W(a)u_{\hbar} \rnor_{L^2(\Rm^n)} = O(\hbar^{\infty}).
\]
This set is linked to the FBI transform $\fbi$, which we will use to define the Berezin-Toeplitz operators, see Proposition \ref{prop_link_pseudo_toeplitz}. Let $u_{\hbar}\in L^2(\Rm^n)$, then $(x,\xi)\notin\wfs_{\hbar}(u)$ if and only if there exists a neighbourhood $W$ of $(x,\xi)$ in $\Rm^{2d}$ such that

\[
\lnor \fbi u_{\hbar} \rnor_{L^2(W)} = O(\hbar^{\infty}).
\]
See, for example, \cite{zwor12} Theorem 13.14. In fact, $\fbi$ is an isometry from $L^2(\Rm^n)$ to the Hilbert space

\[
\fbarg = L^2(\Cm^n)\cap \lacc e^{-\frac{N|z|^2}{2}}f \;/\; f\text{ is holomorphic} \racc,
\]
equipped with the norm of $L^2(\Cm^n)$, called the Bargmann space. Here we denote $N=\frac{1}{\hbar}$. Actually, $\fbarg$ is the space of states on which we define Berezin-Toeplitz operators. Denoting $\Pi_N$ the orthogonal projection from $L^2(\Cm^n)$ to $\fbarg$, for all $f$ in a good symbol space, the Berezin-Toeplitz operator with symbol $f$ is
\begin{align}
\label{eq_def_intro}
T_N(f) : D & \rightarrow \fbarg\\
			u & \mapsto \Pi_N(fu),\nonumber
\end{align}
with $D$ its domain, see Section \ref{sec_Fock}. Moreover, these operators are linked to pseudodifferential operators through the FBI transform. For such an operator, there exists a function $a_{\hbar}$ such that

\[
Op_{\hbar}^W(a_{\hbar}) = \fbi^* T_N(f) \fbi,
\]
see Proposition \ref{prop_link_pseudo_toeplitz}. Berezin-Toeplitz operators can also be constructed on a compact manifold with a definition akin to \eqref{eq_def_intro} where $\fbarg$ is replaced by a suitable subspace of holomorphic objects on the manifold, see Section \ref{sec_quantum} or \cite{bord94}.

With these definitions, the main result of this article is the following:

\begin{theorem}
\label{th_intro}
Let $n\in\Nm$, and $M$ be either $\Cm^n$ or a compact, quantizable, Kähler manifold of dimension $n$. Let $f\in C^{\infty}(M,\Rm)$, if $M=\Cm^n$, we suppose that $f$ and all its derivatives grow at most polynomially, and that $|f(z)|\xrightarrow[|z|\rightarrow +\infty]{} +\infty$. Let $E$ be a regular value of $f$: $f(z)=E$ implies $df(z)\neq 0$. Let $\mu_N$ be a sequence of numbers such that $\mu_N \xrightarrow[N\rightarrow +\infty]{} E$, and $V_N$ a sequence of associated quasimodes
\[
T_N(f) V_N = \mu_N V_N + O_{L^2(M)}(N^{-\infty})
\]
with unit $L^2$ norms. Then for all $N\in\mathbb{N}$ and $p\in[2,+\infty]$
\begin{equation}
\label{eq_res_intro}
\|V_N\|_{L^p(M)} = O\left(N^{\left(n-\frac{1}{2}\right)\left(\frac{1}{2}-\frac{1}{p}\right)}\right).
\end{equation}
\end{theorem}
Furthermore, using the link between Toeplitz and pseudodifferential operators, we prove a similar bound on the FBI transform of quasimodes of pseudodifferential operators. The estimates \eqref{eq_res_intro} are sharp as we can find operators for which the bound becomes an asymptotic equivalence as $N\to\infty$ on $\Cm^n$, see Corollary \ref{cor_harmonic}, and on the complex projective space of dimension $n$, see Theorem \ref{th_optim_proj}.

As in the case of $L^p$ bounds for the Laplace-Beltrami operator's quasimodes, the examples saturating \eqref{eq_res_intro} involve a completely periodic flow on the energy shell. It would certainly be interesting to improve the bounds of Theorem \ref{th_intro} under opposite dynamical assumptions such as Anosov flows. On this topic, some references give pointwise Weyl laws, which imply concentration estimates since they consist of bounding $\sum_{j\in J} |e_j(x)|^2$ for $(e_j)_{j\in J}$ a finite number of eigenfunctions of the operator. We can cite Bérard \cite{bera77} who proved a pointwise Weyl law for the Laplace operator on a compact manifold $M$ with non-positive sectional curvature. Bonthonneau \cite{bont17} generalised this method for non-compact manifold with cusps. Besides, Canzani and Galkowski \cite{canz23} studied the eigenfunctions of the Laplace operator on a compact manifold without boundary. They proved estimates of $L^p$ norms over open sets with respect to $L^2$ norms and terms depending on geodesic flows. Duistermaat and Guillemin \cite{duis75} and Safarov and Vassilev \cite{safa97} proved a pointwise Weyl law for positive elliptic self-adjoint operators on compact manifolds without boundary, with a non-periodic hamiltonian flow hypothesis.

The article will be organised as follows, in Section \ref{sec_Fock} we recall the construction of the Bargmann space of dimension $n$. We also highlight a Hilbert basis $(e_{\alpha})_{\alpha\in\Nm^n}$ which will be used later.  In a second time we remind the Toeplitz quantisation on this space with a class of continuous symbols. Then, as an introduction to concentration estimates of quasimodes of Toeplitz operators, we study the harmonic oscillator. In this quantisation, this operator has purely discrete spectrum given by

\[
\lacc \frac{k+n}{N} /\; k\in\Nm\racc,
\]
and for $k\in\Nm$ the space of eigenfunctions with eigenvalue $\frac{k+n}{N}$ has basis $(e_{\alpha})_{|\alpha|=k}$. In that sense, the functions $e_{\alpha}$ play the same role as the Hermite functions on $\Rm^n$.

In Section \ref{sec_quantum}, we consider Toeplitz operators on Kähler manifolds, and the main purpose is to describe their quantum propagator, which is the main ingredient to prove Theorem \ref{th_intro}. To that end, we describe how Toeplitz operators can be written as Lagrangian states, as developed by Charles \cite{char03}. Finally, we give a result by Charles and Le Floch \cite{char21} that writes the quantum propagator in term of Lagrangian states.

In Section \ref{sec_Lp}, using the Lagrangian states description and another result from \cite{char21}, we give a general expression of spectral projectors of Toeplitz operators, which are operators of the form $\rho\lpar N\lpar E-T_N(f)\rpar\rpar$ with $\rho\in C^{\infty}_c(\Cm)$ and $E\in\Rm$. We use this expression to prove the $L^{\infty}$ estimate of quasimodes of Toeplitz operators on Kähler manifolds. The estimate \eqref{eq_res_intro} is then obtained by $L^p$-$L^q$ interpolation. Furthermore, we prove the sharpness of the Theorem by providing examples in the projective space that saturate \eqref{eq_res_intro}. We insist on the fact that, due to the absence of caustic, this result is simpler than the ones we cited before for pseudodifferantial operators. In \cite{sogg88} and \cite{koch05}, we saw that the estimates obtained by $L^p$-$L^q$ interpolation were not sharp, as shown in Figure \ref{fig_Lp}.

Finally, we use the symbolic calculus for Toeplitz operators on $\Cm$ to express the link between their quasimodes and quasimodes of Toeplitz operators on the torus. This enables us to prove Theorem \ref{th_intro} for flat spaces. Finally, we recall how the FBI transform relates Toeplitz operators on flat spaces to pseudodifferential operators to get a partial estimate on pseudifferential operators.

\subsection*{Notations}

In the next sections, we work on $\mathbb{R}^n$ or $\mathbb{C}^n$ depending on the situation, each with the associated Euclidean and Hermitian norm respectively. Thus, we fix the following convention on notations, for all $(x,y)\in\mathbb{R}^n$, $z\in \mathbb{C}^n, \alpha\in \mathbb{N}^n$

\begin{itemize}
\item $x\cdot y = \sum_{1\le j\le n} x_j y_j,$
\item $(x-y)^2=(x-y)\cdot(x-y),$
\item $|z|^2=z\cdot \overline{z},$
\item $z^2=z\cdot z,$
\item $z^\alpha = \prod_{1\le j\le n} z_j^{\alpha_j}.$
\end{itemize}

We use the following notations for functions spaces, for all $(f,g)\in L^2(\Rm^n)$, $\langle f,g\rangle = \int_{\Rm^n}\overline{f}(x)g(x)dx,$, while for $(f,g)\in S'(\Rm^n)\times S(\Rm^n)$ we will denote by $(f,g)$ the bilinear duality product.

\subsection*{Acknowledgements}

This work was supported by the ANR-24-CE40-5905-01 “STENTOR” project. The author thanks Alix Deleporte and San Vũ Ng\d{o}c for their guidance throughout this project and their proof-readings.

\section{Elementary study on the flat space}
\label{sec_Fock}

In this section, we first recall the construction of the Bargmann space, as well as Toeplitz operators on this space. Then we will study the $L^p$ norms of a basis of eigenfunctions of the Berezin-Toeplitz harmonic oscillator, and in particular we give a general upper bound that becomes an equivalence for a subfamily. In fact, we will see in Section \ref{sec_Lp} that the bound we get in this example applies for quasimodes of a large class of Toeplitz operators on $\Cm^n$.

\subsection{Bargmann space}

The construction of Berezin-Toeplitz operators on $\Cm^n$ relies on the definition of the Hilbert space of state functions; namely, the Bargmann or Segal-Bargmann space. This space first appeared in an article by Fock \cite{fock28}, as the space of functions in which lies the solution to the Dirac statistical equation. Later, Bargmann \cite{barg61} made a mathematical description of this space. Segal \cite{sega63}  used it in the context of representations in the free fields theory and the link between this space and microlocal analysis was established by Folland \cite{foll89} (see also \cite{mart02}) with a description using the space of square integrable functions on $\Rm^n$ and wave packets.

We recall the essential definitions and results here. For the sake of completeness we wrote a detailed construction of the Bargmann space in Section \ref{sec_constr} based on the works of Bargmann \cite{barg61}, Folland \cite{foll89}, Martinez \cite{mart02}, and Zworski \cite{zwor12}.

\begin{definition}
For $N\in\Nm\backslash\{0\}$ the Bargmann space is given by
\[
\fbarg = L^2(\Cm^n)\cap \lacc e^{-\frac{N|z|^2}{2}}f \;/\; f\text{ is holomorphic} \racc,
\]
equipped with the usual $L^2$ scalar product, associated with the Lebesgue measure on $\Rm^{2n}$.
\end{definition}

\begin{proposition}[\cite{barg61} Chapter 1.b]
\label{prop_base}
The family $\left(e_{\alpha} = \frac{N^{\frac{n+|\alpha|}{2}}e^{-\frac{N|z|^2}{2}}z^{\alpha}}{\pi^{\frac{n}{2}}\sqrt{\alpha !}}\right)_{\alpha \in \mathbb{N}^n}$ is a Hilbert basis of $\fbarg$.
\end{proposition}

\begin{proof}
Let $\alpha,\beta \in \mathbb{N}^n$, then $e_{\alpha},e_{\beta} \in \fbarg$ and with a change of variables to polar coordinates, we compute

\begin{align*}
\langle e^{-\frac{N|z|^2}{2}}z^{\alpha},e^{-\frac{N|z|^2}{2}}z^{\beta}\rangle & = \int_{\mathbb{C}^n} e^{-N|z|^2}\overline{z}^{\alpha}z^{\beta},\\
		& = \prod\limits_{1\le j\le n} \int_0^{+\infty}\int_0^{2\pi}e^{-Nr^2}z^{\alpha_j+\beta_j+1}e^{i\theta (\beta_j-\alpha_j)}d\theta dr,\\
		& = \prod\limits_{1\le j\le n} \int_0^{+\infty} e^{-Nr^2}z^{\alpha_j+\beta_j+1} dr \; 2\pi \delta_{\alpha_j,\beta_j},\\
		& = \prod\limits_{1\le j\le n} \frac{\alpha_j! \pi \delta_{\alpha_j,\beta_j}}{N^{\alpha_j+1}},\\
		& = \frac{\alpha! \pi^n \delta_{\alpha,\beta}}{N^{|\alpha |+n}}.
\end{align*}

The family is orthonormal in $\fbarg$, it remains to show that it is total. If $f$ is holomorphic on $\mathbb{C}^n$ then there exists $a_{\alpha} \in \mathbb{C}$ such that $f(z) = \sum\limits_{\alpha\in\mathbb{N}^n} a_{\alpha}z^{\alpha}$ uniformly on every compact of $\mathbb{C}^n$. We can now compute for $\sigma>0$ the truncated $L^2$ norm,

\begin{align*}
M(\sigma)   & = \int_{|z|<\sigma}|f(z)|^2e^{-N|z|^2}dz,\\
			& = \sum\limits_{\alpha ,\beta\in\mathbb{N}^n}\overline{a}_{\alpha}a_{\beta} \prod\limits_{1\le j\le n}\int_0^{\sigma}\int_0^{2\pi}e^{-Nr^2}r^{\alpha_j+\beta_j+1}e^{i\theta (\beta_j-\alpha_j)}d\theta dr,\\
			& = \sum\limits_{\alpha\in\mathbb{N}^n} \frac{|a_{\alpha}|^2\pi^n\alpha !\gamma_{\alpha}(\sigma)}{N^{|\alpha|+n}},
\end{align*}
where for all $\alpha\in\mathbb{N}^n$, the function $\gamma_{\alpha} : [0,\infty[ \mapsto [0,1[$ is continuous increasing and goes to $1$ as $\sigma \rightarrow +\infty$.

If $e^{-\frac{N|\cdot |^2}{2}}f\in\fbarg$, then $M(\sigma) \le \lnor e^{-\frac{N|\cdot |^2}{2}} f \rnor^2_{L^2(\mathbb{C}^n)} < +\infty$ so by monotone convergence $M(\sigma) \xrightarrow{\sigma \rightarrow +\infty} \sum\limits_{\alpha\in\mathbb{N}^n} \frac{|a_{\alpha}|^2\pi^n\alpha}{N^{|\alpha|+n}}$.
What's more, $M(\sigma ) \xrightarrow{\sigma \rightarrow +\infty} \|e^{-\frac{N|\cdot |^2}{2}}f\|^2_{L^2(\mathbb{C}^n)}$, and for $\alpha\in\mathbb{N}^n$ by a similar computation $\langle e^{-\frac{N|\cdot |^2}{2}}f,e_{\alpha}\rangle = \frac{\overline{a}_{\alpha}\pi^{\frac{n}{2}}\sqrt{\alpha !}}{N^{\frac{|\alpha | +n}{2}}}$, thus

\[
\lnor e^{-\frac{N|\cdot |^2}{2}}f \rnor^2_{L^2(\mathbb{C}^n)} = \sum\limits_{\alpha\in\mathbb{N}^n} \frac{|a_{\alpha}|^2\pi^n\alpha}{N^{|\alpha|+n}} = \sum\limits_{\alpha\in\mathbb{N}^n}|\langle e^{-\frac{N|\cdot |^2}{2}}f,e_{\alpha}\rangle|^2.
\]

The family is total in $\fbarg$, it is a Hilbert basis. 
\end{proof}

The orthogonal projection on $\fbarg$ is then given by

\[
\Pi_N f(z) = \sum\limits_{\alpha\in\Nm^n} e_{\alpha}(z) \langle e_{\alpha},f\rangle
\]
for $f\in L^2(\Cm^n)$, which gives the following formula after computation.
\begin{proposition}
The application
\begin{align*}
\Pi_N : L^2(\mathbb{C}^n) & \rightarrow \fbarg\\
f & \mapsto \lpar z \mapsto \lpar \frac{N}{\pi}\rpar^n \int_{\Cm^n} e^{-\frac{N|z|^2}{2}} e^{-\frac{N|w|^2}{2}} e^{Nz\cdot \overline{w}} f(w) dw \rpar
\end{align*}
is the orthogonal projection on $\fbarg$, called the Bergman projection.
\end{proposition}

\subsection{Toeplitz quantization}

We can now review the Toeplitz quantization on $\Cm^n$. In order to fix the notations, we first define our space of symbols.

\begin{definition}
We call \emph{symbol} a $C^0$ function of at most polynomial growth at infinity. We write the symbol space
\[
S = \bigcup_{k\in\mathbb{N}}S_k,
\]
where for all $k\in\mathbb{N}$
\[
S_k = \lacc f\in C^0(\mathbb{C}^n) \; /\;\forall z\in\mathbb{C}^n \; |f(z)| \le (1+|z|)^k\racc.
\]
\end{definition}

Given $f\in S_k$, let us prove that $g\mapsto \Pi_N (fg)$ is well-defined on the following domain
\[
A_k = \lacc g\in\fbarg /\forall \alpha\in\mathbb{N}^n \; |\alpha |=k \Rightarrow z^{\alpha}g\in L^2\racc.
\]
First, we characterise these spaces with

\begin{lemma}
\label{lem_eq_symbol_space}
The $\mathbb{C}$-linear map $\Phi:\fbarg\rightarrow\ell^2(\Nm^n)$ defined on the Hilbert basis by
\[
\Phi(e_{\nu}) = \delta_{\nu}
\]
is an isometry, where $\delta_{\nu}$ is the element of $\ell^2(\Nm^n)$ with all the coefficients equal to 0 except the one of index $\nu$ which is equal to $1$. Furthermore, for all $k\in\mathbb{N}$
\[
\Phi(A_k) = \lacc b\in\ell^2(\mathbb{N}^n)/\forall \alpha\in\mathbb{N}^n \; |\alpha |=k \Rightarrow (\sqrt{n^{\alpha}}b_n)_{n\in\mathbb{N}^n}\in\ell^2(\mathbb{N}^n)\racc
\]
In particular we can write,
\begin{align*}
A_k &= \lacc g\in\fbarg /\forall \alpha\in\mathbb{N}^n \; |\alpha |\le k \Rightarrow z^{\alpha}g\in L^2\racc,\\
	&= \lacc g\in\fbarg /\forall j\in\mathbb{N} \; j\le k \Rightarrow |z|^jg\in L^2\racc.
\end{align*}
\end{lemma}

\begin{proof}
$\Phi$ sends the Hilbert basis of $\fbarg$ on the one of $\ell^2(\mathbb{N}^n)$, which characterise an isometry.
What's more, for all $\alpha,\nu\in\mathbb{N}^n$,
\[
z^{\alpha}e_{\nu} = N^{-\frac{|\alpha |}{2}}\sqrt{\frac{(\nu +\alpha)!}{\nu!}}e_{\nu +\alpha},
\]
with
\[
\sqrt{\frac{(\nu +\alpha)!}{\nu!}} \underset{|\nu|\to +\infty}{\sim} \sqrt{\nu^{\alpha}}.
\]
So for all $f\in\fbarg$ by decomposing it in the Hilbert basis $f = \sum\limits_{\nu\in\mathbb{N}^n} b_{\nu}e_{\nu}$ with $b=\Phi(f)$, we get that for all $\alpha\in\mathbb{N}^n$ with $|\alpha| \ge k$
\begin{align*}
f\in A_k & \Leftrightarrow \sum\limits_{\nu\in\mathbb{N}^n} b_{\nu} N^{-\frac{|\alpha |}{2}}\sqrt{\frac{(\nu +\alpha)!}{\nu!}} e_{\nu+\alpha} \in L^2(\mathbb{C}^n),\\
		& \Leftrightarrow \sqrt{\frac{(\nu +\alpha)!}{\nu!}} b_{\nu} \in l^2(\mathbb{N}^n),\\
		& \Leftrightarrow \sqrt{\nu^{\alpha}} b_{\nu} \in l^2(\mathbb{N}^n).
\end{align*}
Which proves the equality of the two sets. The equivalent definitions are then consequences of the similar ones in $l^2(\mathbb{N}^n)$.
\end{proof}

\begin{proposition}
\label{prop_criter_adjoint}
For all $k\in\mathbb{N}$ and $f\in S_k$ the operator
\begin{align*}
T_N(f): A_k &\rightarrow \fbarg\\
			g &\mapsto \Pi_n(fg)
\end{align*}
is well-defined. Furthermore, if $f$ is real-valued bounded below by $M\in\Rm$ and such that $|f(z)|\ge C|z|$ for all $z\in\mathbb{C}^n$ with $C>0$, then $T_N(f)$ is self-adjoint with compact resolvent.
\end{proposition}

Many references already proved similar results on Toeplitz operators, see for instance \cite{bout81}. We wanted to write Proposition \ref{prop_criter_adjoint} with our notations, as it will be useful later.

\begin{proof}
Let $k\in\mathbb{N}$, $f\in S_k$ and $g\in A_k$ then
\[
|fg|=|f|\, |g|\le (1+|z|)^k|g|\le\sum\limits_{j\le k} \binom{k}{j} |z|^j\, |g| \in L^2,
\]
we can thus apply $\Pi_n$ to $fg$. Hence, $g\mapsto \Pi_n(fg)$ is well-defined on $A_k$. Now, if $f$ is real-valued, bounded below and such that $|f(z)|\ge C|z|$, we define the following quadratic form on $\fbarg$
\begin{align*}
q : A_k \times A_k & \rightarrow \mathbb{C}\\
		(u,v) & \mapsto \langle u,fv\rangle_{L^2(\mathbb{C}^n)} = \langle u,T_N(f)(v)\rangle_{L^2(\mathbb{C}^n)}.
\end{align*}
By hypothesis, $q$ is symmetric, and bounded below by $M$. Furthermore, the norm
\[
\|u\|_k^2 = (M+1)\|u\|_{L^2(\mathbb{C}^n)}^2 + \langle u,fu\rangle_{L^2(\mathbb{C}^n)} \le (M+1)\|u\|_{L^2(\mathbb{C}^n)}^2 + \|\sqrt{1+|z|^k}u\|_{L^2(\mathbb{C}^n)}^2
\]
makes $A_k$ complete by definition, that is $q$ is closed. Hence, $T_N(f)$ is self-adjoint (\cite{reed72} theorem VIII.15).
Moreover, for all $b>0$, using the properties of $\Pi_N$ and the hypothesis on $f$
\begin{align*}
		& \lacc u\in A_k / \|u\|_{L^2(\mathbb{C}^n)}\le1,\; \|fu\|_{L^2(\mathbb{C}^n)}\le b\racc\\
\subset & \lacc u\in A_k / \|u\|_{L^2(\mathbb{C}^n)}\le1,\; \||z|^ku\|_{L^2(\mathbb{C}^n)}\le b\racc\\
\subset & \lacc u\in A_k / \|u\|_{L^2(\mathbb{C}^n)}\le1,\; \||z|u\|_{L^2(\mathbb{C}^n)}\le b,\; \|\nabla u\|_{L^2(\mathbb{C}^n)}\le b\racc.
\end{align*}
Yet the first set is by definition closed, while the last one is compact in $L^2$ by the Rellich's criterion (\cite{reed78} theorem XIII.65), thus the first one is compact in $\fbarg$. So $T_N(f)$ has compact resolvent (\cite{reed78} theorem XIII.64).

\end{proof}

\subsection{Harmonic oscillator}

An important example arising in quantum mechanics is the harmonic oscillator $P=T_N(|z|^2)$. For that reason, we compute the $L^p$ norms of its eigenfunctions, which provides insight into the behaviour of Berezin-Toeplitz operators. Even though the Harmonic oscillator has already been extensively studied in many topics, this computation has not been done before to the author's knowledge.

\begin{lemma}
The spectrum of $P$ is made of discrete eigenvalues,
\[
\sigma(P) = \sigma_p(P) = \lacc \frac{k+n}{N} \; /\; k\in\mathbb{N}\racc,
\]
and for all $k\in\mathbb{N}$ the eigenspace of $P$ associated with the eigenvalue $\frac{k+n}{N}$ is spanned by the functions $e_{\alpha}$ with $|\alpha|=k$.
\end{lemma}

\begin{proof}
First, Proposition \ref{prop_criter_adjoint} and the spectral theorem already ensure that the spectrum is discrete, we just have to find the eigenvalues. For $1\le j\le n$, we denote

\begin{align*}
D_j : \fbarg & \rightarrow \fbarg\\
	g & \mapsto \frac{e^{-\frac{N|z|^2}{2}}}{N} \partial_{z_j}\lpar e^{-\frac{N|z|^2}{2}} g \rpar.
\end{align*}
In particular, we have $P=\sum\limits_{1\le j\le n}D_j\circ z_j$, and for $j\in\{ 1,\cdots ,n\}$ and $\alpha\in\mathbb{N}^n$

\begin{align*}
\frac{(\pi)^{\frac{n}{2}}\sqrt{\alpha !}}{N^{\frac{n+|\alpha|}{2}}e^{-\frac{N|z|^2}{2}}}D_j\circ z_j(e_{\alpha})
	&= \frac{\partial_{z_j}(z_j z^{\alpha})}{N},\\
	&= \frac{\alpha_j+1}{N}z^{\alpha},\\
\text{thus } P(e_{\alpha}) &= \sum\limits_{1\le j\le n} \frac{\alpha_j+1}{N}e_{\alpha},\\
				&= \frac{|\alpha|+n}{N}e_{\alpha}.
\end{align*}
Since $e_{\alpha}$ is a Hilbert basis for $\fbarg$ because of Proposition \ref{prop_base}, these are exactly the eigenvalues of $P$, and we found the associated Hilbert basis of eigenfunctions.
\end{proof}

We now consider the eigenvalue $\frac{|\mu|+n}{N}$ fixed modulo $O\lpar\frac{1}{N}\rpar$, and we look for equivalents of the $L^p$ norms when $N$ goes to $+\infty$, or equivalently, when $|\mu|$ goes to $+\infty$. The following computations would work for $p\in[1,+\infty[$, but we are only interested in $p\in[2,+\infty[$.

We first write for $a\in [1,+\infty [$
\[
E(a) = \int_{\mathbb{C}}e^{-N|z|^2}|z^a|dz = 2\pi\int_0^{+\infty}e^{-Nr^2}r^{a+1}dr = \frac{\pi\Gamma \left(\frac{a}{2}+1\right)}{N^{\frac{a}{2}+1}},
\]
so that
\[
e_{\nu} = \frac{e^{-\frac{N|z|^2}{2}}z^{\nu}}{\prod\limits_{1\le j\le n}E(2\nu_j)^{\frac{1}{2}}}.
\]

\begin{proposition}
\label{prop_formula_norms}
Let $\nu\in\mathbb{N}^n$ then for all $p\in[2,+\infty[$,
\begin{align*}
\|e_{\nu}\|_{L^p(\mathbb{C}^n)}			&= \left(\frac{2}{p}\right)^{\frac{|\nu|}{2}+\frac{n}{p}}\prod\limits_{1\le j\le n}\frac{E(p\nu_j)^{\frac{1}{p}}}{E(2\nu_j)^{\frac{1}{2}}},\\
\|e_{\nu}\|_{L^{\infty}(\mathbb{C}^n)}  &= \frac{e^{-\frac{\nu}{2}}\nu^{\frac{\nu}{2}}N^{\frac{n}{2}}}{\pi^{\frac{n}{2}}\sqrt{\nu!}}.
\end{align*}
\end{proposition}

\begin{proof}
First we see that,
\[
\|e_{\nu}\|^p_{L^p(\mathbb{C}^n)} = \prod\limits_{1\le j\le n} E(2\nu_j)^{-\frac{p}{2}} \int_{\mathbb{C}}e^{-\frac{Np}{2}z_j^2}|z_j|^{p\nu_j}dz_j.
\]
Now, for $1\le j\le n$, applying the change of variables $z_j \mapsto \sqrt{\frac{2}{p}}z_j$ we get
\begin{align*}
\int_{\mathbb{C}^n}e^{-\frac{Np}{2}z_j^2}|z_j|^{p\nu_j}dz_j &= \int_{\mathbb{C}^n}e^{-Nz_j^2}\left(\frac{2}{p}\right)^{\frac{p\nu_j}{2}}|z_j|^{p\nu_j}\left(\frac{2}{p}\right)dz_j,\\
	&= \left(\frac{2}{p}\right)^{\frac{p\nu_j}{2}+1}E(p\nu_j),
\end{align*}
hence the formula for a finite $p$. Then, for the $\sup$-norm, we see that the derivative of $|e_{\nu}|$ by $|z_j|$ is,

\[
\partial_{|z_j|} |e_{\nu}| = \lpar \frac{\nu_j}{|z_j|}-N|z_j| \rpar |e_{\nu}|
\]
thus the maximum is attained when $|z_j|=\sqrt{\frac{\nu_j}{N}}$, and is equal to
\[
\|e_{\nu}\|_{L^{\infty}(\mathbb{C}^n)} = \frac{e^{-\frac{\nu}{2}}\nu^{\frac{\nu}{2}}N^{\frac{n}{2}}}{\pi^{\frac{n}{2}}\sqrt{\nu!}}.
\]
\end{proof}

The asymptotic of $\lnor e_{\nu} \rnor_{L^p(\Cm^n)}$ for $|\nu|$ large highly depends on the direction in $\Nm^n$. We now give an argument of $\log$-convexity that will allow the asymptotic study. First we need a Stirling formula written as an equality.

\begin{proposition}[\cite{arti64} Chapter 3]
\label{prop_Arti}
For all $x\ge 1$,
\[
\Gamma (x)= \sqrt{\frac{2\pi}{x}}\left(\frac{x}{e}\right)^xe^{\frac{\theta (x)}{12x}}
\]
where $\theta (x)\in [0,1]$.
\end{proposition}

\begin{corollaryp}
Let $p\in [2,+\infty [$ and $\nu\in\mathbb{N}^n$, then
\begin{equation}
\label{eqnormep}
\|e_{\nu}\|_{L^p(\mathbb{C}^n)} = \prod\limits_{1\le j\le n} \left(\frac{\pi^{\frac{3}{2}}}{Ne}\right)^{\frac{1}{p}-\frac{1}{2}} e^{\epsilon(p,\nu_j)} \left(\frac{2}{p}\right)^{\frac{1}{p}} \left(\frac{\nu_j+\frac{2}{p}}{\nu_j+1}\right)^{\frac{\nu_j}{2}} \frac{(p\nu_j+2)^{\frac{1}{2p}}}{(2\nu_j+2)^{\frac{1}{4}}}
\end{equation}
where $\epsilon(p,\nu_j) = \frac{\theta\left(\frac{p}{2}\nu_j+1\right)}{6p(p\nu_j+2)}-\frac{\theta(\nu_j+1)}{24(\nu_j+1)}\in[-\frac{1}{24},\frac{1}{24}]$, and
\[
\|e_{\nu}\|_{L^{\infty}(\mathbb{C}^n)} = \prod\limits_{1\le j\le n} \left(\frac{Ne}{\pi^{\frac{3}{2}}}\right)^{\frac{1}{2}} e^{-\frac{\theta(\nu_j+1)}{24(\nu_j+1)}} \left(\frac{\nu_j}{\nu_j+1}\right)^{\frac{\nu_j}{2}} \frac{1}{(2\nu_j+2)^{\frac{1}{4}}}
\]
\end{corollaryp}

\begin{proof}
Using Propositions \ref{prop_formula_norms} and \ref{prop_Arti} we get that if $a\ge 0$ then
\[
E(a) = \frac{2\pi^{\frac{3}{2}}}{N^{\frac{a}{2}+1}} (a+2)^{\frac{a+1}{2}} (2e)^{-\frac{a+2}{2}} e^{\frac{\theta\left(\frac{a}{2}+1\right)}{6(a+2))}},
\]
in particular if $\mu_1 \in \mathbb{N}$,

\[
E(p\nu_1)^{\frac{1}{p}} = \frac{\pi^{\frac{3}{2p}}}{N^{\frac{\nu_1}{2}+\frac{1}{p}}} (p\nu_1+2)^{\frac{\nu_1}{2}+\frac{1}{2p}} 2^{-\frac{\nu_1}{2}} e^{-\frac{\nu_1}{2}-\frac{1}{p}} e^{\frac{\theta\left(\frac{p}{2}\nu_1+1\right)}{6p(p\nu_1+2)}}.
\]
Hence,

\[
\frac{E(p\nu_1)^{\frac{1}{p}}}{E(2\nu_1)^{\frac{1}{2}}} = \left(\frac{\pi^{\frac{3}{2}}}{Ne}\right)^{\frac{1}{p}-\frac{1}{2}} e^{\epsilon(p,\nu_1)} \left(\frac{\nu_1+\frac{2}{p}}{\nu_1+1}\right)^{\frac{\nu_1}{2}} \left(\frac{p}{2}\right)^{\frac{\nu_1}{2}} \frac{(p\nu_1+2)^{\frac{1}{2p}}}{(2\nu_1+2)^{\frac{1}{4}}},
\]
and we obtain formula \eqref{eqnormep} by replacing this expression in

\[
\|e_{\nu}\|_{L^p(\mathbb{C}^n)} = \prod\limits_{1\le j\le n} \left(\frac{2}{p}\right)^{\frac{\nu_j}{2}+\frac{1}{p}} \frac{E(p\nu_j)^{\frac{1}{p}}}{E(2\nu_j)^{\frac{1}{2}}}.
\]
The same reasoning applies for the $L^{\infty}$ norm by replacing $\sqrt{\nu !}$ by
\[
\prod\limits_{1\le j\le n} \Gamma(\nu_j+1)^{\frac{1}{2}} = \left(\frac{2\pi}{\nu_j+1}\right)^{\frac{1}{4}} \left(\frac{\nu_j+1}{e}\right)^{\frac{\nu_j+1}{2}} e^{\frac{\theta(\nu_j+1)}{24(\nu_j+1)}}
\]
\end{proof}

The full expression of the norm is unfortunately not $\log$-convex with respect to $\nu$, though it is the case for the last two terms in the product of \eqref{eqnormep}.

\begin{proposition}
\label{prop_convex}
The functions
\begin{align*}
f \;:\; [0,+\infty[ & \rightarrow \mathbb{R}\\
			x & \mapsto \left(\frac{x+\frac{2}{p}}{x+1}\right)^{\frac{x}{2}} \frac{(px+2)^{\frac{1}{2p}}}{(2x+2)^{\frac{1}{4}}},
\end{align*}
and
\begin{align*}
g \;:\; [0,+\infty[ & \rightarrow \mathbb{R}\\
			x & \mapsto \left(\frac{x}{x+1}\right)^{\frac{x}{2}} \frac{1}{(2x+2)^{\frac{1}{4}}}
\end{align*}
are $\log$ convex.
\end{proposition}

\begin{proof}
For all $x\in[0,+\infty[$,
\begin{align*}
\log(f)(x) = & \frac{x}{2}\log \lpar x+\frac{2}{p} \rpar -\frac{x}{2}\log(x+1)+\frac{1}{2p}\log(px+2)-\frac{1}{4}\log(2x+2)\\
\log(f)'(x) = & \frac{1}{2}\log \lpar x+\frac{2}{p} \rpar-\frac{1}{2}\log(x+1)\\
			& +\frac{x}{2 \lpar x+\frac{2}{p} \rpar }-\frac{x}{2(x+1)} + \frac{1}{2(px+2)}-\frac{1}{2(2x+2)}\\
			= & \frac{1}{2} \lpar \log \lpar x+\frac{2}{p} \rpar-\log(x+1) \rpar+\frac{px+1}{2(px+2)}-\frac{2x+1}{2(2x+2)}\\
\log(f)''(x) = & \frac{p}{2(px+2)}-\frac{1}{2x+2}+\frac{p}{2(px+2)^2}-\frac{1}{(2x+2)^2}\\
			= & \frac{1}{2(px+2)^2(2x+2)^2} \times\\
			  & (p(px+3)(4x^2+8x+4)-2(2x+3)(p^2x^2+4px+4))\\
			= & \frac{p-2}{(px+2)^2(2x+2)^2}(px^2+2(p+2)x+6)
\end{align*}
The polynomial at the numerator is positive for a positive $x$, hence $\log(f)$ is convex.

The same argument works for $g$, for all $x\in[0,+\infty[$
\begin{align*}
\log(g)(x) &= \frac{x}{2}\log(x)-\frac{x}{2}\log(x+1)-\frac{1}{4}\log(2x+2)\\
\log(g)'(x) &= \frac{1}{2}\log(x)-\frac{1}{2}\log(x+1)+\frac{1}{2}-\frac{x}{2(x+1)}-\frac{1}{2(2x+2)}\\
			&= \frac{1}{2}(\log(x)-\log(x+1))+\frac{1}{2}-\frac{2x+1}{2(2x+2)}\\
\log(g)''(x) &= \frac{1}{2x}-\frac{1}{2x+2}-\frac{1}{(2x+2)^2}\\
			&= \frac{1}{2x(2x+2)^2}((2x+2)^2-2x(2x+2)-2x)\\
			&= \frac{1}{x(2x+2)^2}(x+2)
\end{align*}
hence $\log(g)$ is convex.
\end{proof}

\begin{corollaryp}
\label{cor_harmonic}
Fix $\lambda>0$ and consider $k\in\Nm$ depending on $N$ such that $k\underset{N\to\infty}{\sim}\lambda N$. Then for all $p\in[2,+\infty]$ there exists $C>0$, depending only on $n$, $p$ and $\lambda$, such that for all $\nu\in\mathbb{N}^n$ with $|\nu|=k$
\[
\|e_{\nu}\|_{L^p(\mathbb{C}^n)} \le C N^{\left(n-\frac{1}{2}\right)\left(\frac{1}{2}-\frac{1}{p}\right)}.
\]
Furthermore, this bound is sharp since the inequality becomes an equivalence when $\nu$ is replaced by $\mu = (k,0,\cdots,0)\in\mathbb{N}^n$.

\end{corollaryp}

\begin{proof}
We first prove that the $L^p(\mathbb{C}^n)$ norm of $e_{\nu}$ is bounded by the norm of $e_{\mu}$, then we will prove the equivalence for this second term. Using formula \eqref{eqnormep} and Proposition \ref{prop_convex},
\begin{align*}
\|e_{\nu}\|_{L^p(\mathbb{C}^n)} &= \left(\frac{\pi^{\frac{3}{2}}}{Ne}\right)^{n\left(\frac{1}{p}-\frac{1}{2}\right)} \left(\frac{p}{2}\right)^{-\frac{n}{p}} \prod\limits_{1\le j\le n} e^{\epsilon(p,\nu_j)} \prod\limits_{1\le j\le n} f(\nu_j)\\
								&\le \left(\frac{\pi^{\frac{3}{2}}}{Ne}\right)^{n\left(\frac{1}{p}-\frac{1}{2}\right)} \left(\frac{p}{2}\right)^{-\frac{n}{p}} \prod\limits_{1\le j\le n} e^{\epsilon(p,\nu_j)} f(|\nu|)f(0)^{n-1}\\
								&\le \|e_{\mu}\|_{L^p(\mathbb{C}^n)} \prod\limits_{1\le j\le n}e^{\epsilon(p,\nu_j)-\epsilon(p,|\nu|)} \; e^{-(n-1)\epsilon(p,0)}\\
								&\le \|e_{\mu}\|_{L^p(\mathbb{C}^n)} \prod\limits_{1\le j\le n}e^{\epsilon(p,\nu_j)-\epsilon(p,|\nu|)} \; e^{-(n-1)\epsilon(p,0)} .
\end{align*}
We then use the following bounds $\epsilon(p,\nu_j) - \epsilon(p,|\nu|) \le \frac{1}{12}$, $\epsilon(p,0) \le \frac{1}{24}$ and by computing $E(0)^{\frac{1}{p}-\frac{1}{2}}$ we get
\[
e^{-\epsilon(p,0)} = \left(\frac{e}{\sqrt{2\pi}}\right)^{\frac{1}{2}-\frac{1}{p}} \le \sqrt{\frac{e}{\sqrt{2\pi}}},
\]
and put together it gives,
\begin{equation}
\label{eq_numu}
\|e_{\nu}\|_{L^p(\mathbb{C}^n)} \le e^{\frac{1}{24}} \left(e^{\frac{1}{12}}\sqrt{\frac{e}{\sqrt{2\pi}}}\right)^n \|e_{\mu}\|_{L^p(\mathbb{C}^n)}.
\end{equation}

The problem now comes down to finding an asymptotic equivalent of $e_{\mu}$. Its norm is given by

\begin{multline*}
\|e_{\mu}\|_{L^p(\mathbb{C}^n)} =
 \left(\frac{\pi^{\frac{3}{2}}}{Ne}\right)^{\frac{1}{p}-\frac{1}{2}}
e^{\epsilon(p,|\mu|)}
\left(\frac{2}{p}\right)^{\frac{1}{p}}
\left(\frac{|\mu|+\frac{2}{p}}{|\mu|+1}\right)^{\frac{|\mu|}{2}} \\ \times
\frac{(p|\mu|+2)^{\frac{1}{2p}}}{(2|\mu|+2)^{\frac{1}{4}}}
\left(\frac{2}{p}\right)^{\frac{n-1}{p}}
\left(\frac{\pi}{N}\right)^{(n-1)\left(\frac{1}{p}-\frac{1}{2}\right)}.
\end{multline*}
Then using a series expansion at order 1,
\begin{align*}
\left(\frac{|\mu|+\frac{2}{p}}{|\mu|+1}\right)^{\frac{|\mu|}{2}}
	& = \exp\left(\frac{|\mu|}{2}\left(\log\left(1+\frac{2}{p|\mu|}\right)-\log\left(1+\frac{1}{|\mu|}\right)\right)\right)\\
	& \underset{|\mu|\to +\infty}{\sim} e^{\frac{1}{p}-\frac{1}{2}} e^{\mathrm{o}\left(\frac{1}{p|\mu|}\right)} \underset{|\mu|\rightarrow +\infty}{\sim} e^{\frac{1}{p}-\frac{1}{2}}.
\end{align*}
So,
\begin{align*}
\|e_{\mu}\|_{L^p(\mathbb{C}^n)} & \underset{|\mu|\to +\infty}{\sim} C N^{\frac{1}{2}-\frac{1}{p}} |\mu|^{\frac{1}{2}\left(\frac{1}{p}-\frac{1}{2}\right)} N^{(n-1)\left(\frac{1}{2}-\frac{1}{p}\right)}\\
	& \underset{|\mu|\to +\infty}{\sim} C \lambda^{n\left(\frac{1}{p}-\frac{1}{2}\right)} |\mu|^{(n-\frac{1}{2})\left(\frac{1}{2}-\frac{1}{p}\right)}
\end{align*}
with $C$ depending on $p$ and $n$ only, giving the result for $p<\infty$. We do the same for $p=\infty$,
\begin{align*}
\|e_{\nu}\|_{L^{\infty}(\mathbb{C}^n)} & \le e^{\frac{\theta(|\nu|+1)}{24(|\nu|+1)}} e^{(n-1)\frac{\theta(1)}{24}} \prod\limits_{1\le j\le n} e^{-\frac{\theta(\nu_j+1)}{24(\nu_j+1)}} \|e_{\mu}\|_{L^{\infty}(\mathbb{C}^n)}\\
	& \le e^{\frac{\theta(|\nu|+1)}{24(|\nu|+1)}-\frac{\theta(1)}{24}} \left(e^{\frac{1}{24}}\sqrt{\frac{e}{\sqrt{2\pi}}}\right)^n \|e_{\mu}\|_{L^{\infty}(\mathbb{C}^n)}.
\end{align*}
In the same way we compute
\[
\|e_{\mu}\|_{L^{\infty}(\mathbb{C}^n)} =
\left(\frac{\pi^{\frac{3}{2}}}{Ne}\right)^{-\frac{1}{2}}
e^{-\frac{\theta(|\mu|+1}{24(|\mu|+1)}}
\left(\frac{|\mu|}{|\mu|+1}\right)^{\frac{|\mu|}{2}}
\frac{1}{(2|\mu|+2)^{\frac{1}{4}}}
\left(\frac{\pi}{N}\right)^{-\frac{1}{2}(n-1)},
\]
thus
\[
\|e_{\mu}\|_{L^{\infty}(\mathbb{C}^n)} \underset{|\mu|\to +\infty}{\sim} C \lambda^{-\frac{n}{2}} |\mu|^{\left(n-\frac{1}{2}\right)\frac{1}{2}}.
\]
We deduce the result from the fact that $|\nu|$ and $\lambda N$ are equivalent, since $|\nu| = k$.
\end{proof}

In equation \eqref{eq_numu}, the constant at the right-hand side depends on the estimate from Proposition \ref{prop_Arti}, hence it might be possible to get a constant closer to $1$ with a more accurate Stirling formula.

Even if the family $(e_{\nu})_{\nu\in\Nm^n}$ spans the set of eigenfunctions of $P$, the eigenfunctions associated with $\frac{k+n}{N}$ for a fixed $k\in\Nm$ are all the linear combinations of the $e_{\nu}$ with $|\nu|=k$. This high multiplicity is the main difficulty for bounding all the eigenfunctions. However, we prove in Section \ref{sec_Lp} a general estimate for quasimodes of Toeplitz operators on a space similar to $\fbarg$, where the exponent is the same as in Corollary \ref{cor_harmonic}. Despite this difficulty, it is possible to get more refined estimates for specific multi-indices $\mu$.

\begin{proposition}
Consider $1\le \alpha \le n$ an integer, we write for $k\in\mathbb{N}$,
\[
\nu_k = \lpar \frac{k}{\alpha},\cdots,\frac{k}{\alpha},0,\cdots,0 \rpar
\]
a multi index of size $k$ with $\alpha$ identical non-zero coefficients and $n-\alpha$ zero coefficients. As in the previous theorem, we fix a $\lambda>0$ and consider a sequence of integers $k$ such that $k\underset{N\to\infty}{\sim}\lambda N$. Let $p\in[2,+\infty]$, then there exists $C>0$ depending only on $n$, $p$ and $\lambda$, such that for all $k\in\mathbb{N}$
\[
\|e_{\nu_k}\|_{L^p(\mathbb{C}^n)} \underset{N\to\infty}{\sim} C N^{\left(n-\frac{\alpha}{2}\right)\left(\frac{1}{2}-\frac{1}{p}\right)}.
\] 
\end{proposition}

\begin{proof}
Using equation \eqref{eqnormep}, we get for all $k\in\mathbb{N}$
\begin{align*}
\|e_{\nu_k}\|_{L^p(\mathbb{C}^n)} = & \left(\frac{2}{p}\right)^{\frac{k}{2}-\frac{n}{p}} \left(2\pi^{\frac{3}{2}}\right)^{\alpha\left(\frac{1}{p}-\frac{1}{2}\right)} N^{-\alpha\left(\frac{1}{p}-\frac{1}{2}\right)} \left(\frac{\frac{p}{\alpha}k+2}{\frac{2}{\alpha}k+2}\right)^{\frac{k}{2}}\\
	& \times \frac{\left(\frac{p}{\alpha}k+2\right)^{\frac{\alpha}{2p}}}{\left(\frac{2}{\alpha}k+2\right)^{\frac{\alpha}{4}}} (2e)^{-\alpha\left(\frac{1}{p}-\frac{1}{2}\right)} \left(\frac{\pi}{N}\right)^{\alpha\left(\frac{1}{p}-\frac{1}{2}\right)} e^{E_{\alpha}(p,k)}
\end{align*}
where $E_{\alpha}(p,k) \in \left[-\frac{\alpha^2}{24k},\frac{\alpha^2}{24p^2k}\right]$.
We notice that
\[
\left(\frac{\frac{p}{\alpha}k+2}{\frac{2}{\alpha}k+2}\right)^{\frac{k}{2}} \underset{N\to\infty}{\sim} \left(\frac{p}{2}\right)^{\frac{k}{2}}e^{\alpha\left(\frac{1}{p}-\frac{1}{2}\right)}.
\]
so
\begin{align*}
\|e_{\nu_k}\|_{L^p(\mathbb{C}^n)} & \underset{N\to\infty}{\sim} C N^{-n\left(\frac{1}{p}-\frac{1}{2}\right)} k^{\frac{\alpha}{2}\left(\frac{1}{p}-\frac{1}{2}\right)}\\
		& \underset{N\to\infty}{\sim} C \lambda^{n\left(\frac{1}{p}-\frac{1}{2}\right)}k^{\left(n-\frac{\alpha}{2}\right)\left(\frac{1}{2}-\frac{1}{p}\right)}.
\end{align*}
As previously, we conclude using that $k$ and $N$ are asymptotically equivalent up to the constant $\lambda$.
\end{proof}

This simple example already shows a difference with the bound of Sogge \cite{sogg88} and Koch \cite{koch05}, as the exponent of $N$ depends linearly on $\frac{1}{p}$. 

\section{Semiclassical analysis of Berezin-Toeplitz operators}\label{sec_quantum}

From now on, we will consider compact manifolds, with the proof of Theorem \ref{th_intro} in mind. We will see in Section \ref{sec_Lp} that the result on compact manifolds implies the case of $\Cm^n$. We first recall the construction of Berezin-Toeplitz operators on Kähler manifolds. Then, the main purpose of this section is to explicit the quantum propagator of these operators, which will be crucial in the proof of Theorem \ref{th_intro}. For that purpose, we will recall geometric tools developed by Charles and Le Floch \cite{char21}.

\subsection{Quantization on Kähler manifolds}

The microlocal and semiclassical analysis of Toeplitz operators on Kähler manifolds was initiated in \cite{bout81}, and is developed in \cite{bord94,guil95,char03a,ma08}. Besides, the asymptotic properties of the Bergman kernel are well-studied now, see for instance \cite{bout75,zeld17}. In order to fix the notations before recalling the construction of Berezin-Toeplitz operators, we give the necessary notions of Kähler geometry. The advancement of this section follows \cite{lef18}.

Let $L \rightarrow M$ be a holomorphic line bundle over a compact Kähler manifold of dimension $n$.

\begin{definition}
Let $\nabla$ be a connection on $L$, the curvature of $\nabla$ is the differential form $\curv(\nabla) \in \Omega^2(M)\otimes\mathbb{C}$ such that for all vector fields $X,Y$ and section $u$ of $L$

\[
\curv(\nabla)(X,Y)u = \nabla_X\nabla_Yu - \nabla_Y\nabla_Xu - \nabla_{[X,Y]}u
\]
If the connection can be written locally as $\nabla = d+\beta$ then its curvature is $\curv(\nabla)=d\beta$.
\end{definition}

\begin{proposition}
\label{thconstr}
Let $s$  be a local non-zero holomorphic section of $L$ and write $H=h(s,s)$. Then the curvature of the Chern connection is given by
\[
\curv(\nabla) = \overline{\partial}\partial(\log(H)).
\]
\end{proposition}

We now suppose $M$ to be a Kähler manifold, we recall that $\Omega(M)$ can then be decomposed as $\Omega^{1,0}M \oplus \Omega^{0,1}M$, the holomorphic and anti-holomorphic parts.

\begin{proposition}
The decomposition $\Omega(M) = \Omega^{1,0}M \oplus \Omega^{0,1}M$ implies a decomposition on the connection
\[
\nabla = \nabla^{1,0} + \nabla^{0,1}.
\]
We are then able to characterise the holomorphic structure with the connection
\[
s \text{ is a local holomorphic section of } L\rightarrow M \Leftrightarrow \nabla^{0,1}(s)=0.	
\]
\end{proposition}

\begin{definition}
A pre-quantum line bundle $(L,\nabla,h)\rightarrow M$ is a hermitian complex line bundle whose Chern connection has curvature equal to $-i\omega$.
\end{definition}

From now on, we will consider tensor powers of such a bundle, which we write $L^{\otimes N}$. This new line bundle has an induced metric $h_N$
\[
h_N(p_1\otimes\cdots\otimes p_N,q_1\otimes\cdots\otimes q_N) = \prod\limits_{1\le j\le N} h(p_j,q_j)
\]
and an induced connection $\nabla_N$ such that $\curv(\nabla_N)=N\cdot \curv(\nabla) = -iN\omega$. On the space $\Gamma(L^{\otimes N})$ of smooth sections of $L^{\otimes N}$ we consider the Hermitian product given for all $u,v \in \Gamma(L^{\otimes N})$ by
\[
\langle u,v \rangle_N = \int_M h_N(u,v)\frac{|\omega^n|}{n!}
\]
with associated norm $\|\cdot\|_N$. We then define the quantum space
\[
\mathcal{H}_N = H^0\lpar M,L^{\otimes N}\rpar
\]
of holomorphic sections of $L^{\otimes N}\rightarrow M$. We will sometimes use the notation $\hbar = \frac{1}{N}$, which is the new semi-classical parameter. In particular, the semi-classical limit will correspond to $N\rightarrow +\infty$.

\begin{definition}
Let $L^2(M,L^{\otimes N})$ be the completion of $\Gamma(L^{\otimes N})$ with the scalar product $\langle\cdot,\cdot\rangle_N$, and write $\Pi_N$ the Bergman projection, that is the orthogonal projection from $L^2(M,L^{\otimes N})$ to $\mathcal{H}_N$. For $f\in C^0(M)$, the associated Berezin-Toeplitz operator is
\[
T_N(f) = \Pi_N f = \Pi_N f\Pi_N : \mathcal{H}_N \rightarrow \mathcal{H}_N
\]
\end{definition}

\begin{proposition}
Let $f\in C^0(M)$ then
\begin{itemize}
\item $\|T_N(f)\| \le \|f\|_{L^{\infty}(M)}$ where $\|T_N(f)\|$ is the operator norm on $H_N$.
\item $T_N(f)^* = T_N(\overline{f})$, in particular if $f$ is real-valued, then $T_N(f)$ is self-adjoint.
\end{itemize}
\end{proposition}

\subsection{Lagrangian states}

In order to study the spectral properties of Toeplitz operators, it is convenient to write both the operators and their quasimodes as Lagrangian states. We will use the constructions and notations of Charles and Le Floch \cite{char21}.

Let $L\rightarrow M$ be a pre-quantum line bundle over a Kähler manifold, and $L'\rightarrow M$ be a complex line bundle. From now on, the quantum space is $\mathcal{H}_N = H^0(M,L^{\otimes N}\otimes L')$.

\begin{lemma}[\cite{char03} Chapter 2]
Let $\Gamma$ be a Lagrangian submanifold of $M$. For $x\in \Gamma$, there exists a neighbourhood $U\subset M$ of $x$ and a section $F\; :\; U \rightarrow L$ such that $F_{U\cap \Gamma}$ is flat, unitary and such that for every holomorphic vector field $Z$
\[
\nabla_{\overline{Z}}F=0 \quad mod \; \mathcal{I}^{\infty}(\Gamma\cap U)
\]
where $\mathcal{I}^{\infty}(\Gamma\cap U)$ is the ideal of $C^{\infty}$ functions which vanish to any order along $\Gamma\cap U$. Furthermore, if $F' \; :\; U' \rightarrow L$ satisfies the same hypothesis and if $U\cap U'$ is connected then there exists $a\in\Rm$ such that
\[
e^{ia}F=F' \quad mod \; \mathcal{I}^{\infty}(\Gamma\cap U\cap U').
\]
\end{lemma}

It is also possible to prove that on a neighbourhood of $\Gamma$ inside $U$, $|F|$ is strictly lower than $1$ outside $\Gamma$. Hence, we can modify it outside that neighbourhood such that $|F|<1$ on $U\backslash\Gamma$

Fix a real segment $I$ which will be the domain of the time parameter. Let $\mathbb{C}_I$ be the trivial complex line bundle over $I$, $\Gamma$ a closed submanifold of $I\times M$ and $s\in C^{\infty}(\Gamma,\mathbb{C}_I \boxtimes L)$ such that
\begin{itemize}
\item the application $q \; :\; \Gamma \rightarrow I,\; (t,x) \mapsto t$ is a proper submersion, so that for every $t\in I$, the fibre $\Gamma_t = \Gamma\cap\lpar \{t\}\times M \rpar$ is a submanifold of $M$.
\item For all $t\in I$, $\Gamma_t$ is a Lagrangian submanifold of $M$ and the restriction of $s$ to $\Gamma_t$ is flat and unitary.
\end{itemize}
According to the previous lemma and remark, it is now possible to consider a section $F$ of $C_I \boxtimes L$ such that $F|_{\Gamma} = s$, $\overline{\partial}F$ vanishes at every order on $\Gamma$, and $|F|<1$ outside $\Gamma$.

\begin{definition}[\cite{char21} Chapter 2.1]
\label{def_Lagr_state}
A Lagrangian state family associated to $(\Gamma,s)$ is a family $(\psi_N)\in C^{\infty}(I,\mathcal{H}_N)$ such that for every $k\in\mathbb{N}$,
\[
\psi_N(t,x) = \left(\frac{N}{2\pi}\right)^{\frac{n}{4}}F^N(t,x)\sum\limits_{l=0}^k N^{-l}a_l(t,x)+R_k(t,x,N)
\]
where
\begin{itemize}
\item each $a_l$ is a section of $\mathbb{C}_I \boxtimes L'$ such that $\overline{\partial}a_l$ vanish at every order on $\Gamma$.
\item For all integers $p$ and $k$, $\partial_t^pR_k = O(N^{p-k-1})$ uniformly on every compact of $I\times M$.
\end{itemize}
\end{definition}

In fact, a Lagrangian section is entirely characterised modulo $O(\hbar^{\infty})$ by the values of the section $a_l$ restricted to $\Gamma$. Indeed, according to the Section 2.2 of \cite{char03}, for every family $(b_l)_{l\in\mathbb{N}}\in C^{\infty}(\Gamma,\mathbb{C}_I \boxtimes L')$, there exists a Lagrangian state $\psi_N$ such that for all $y\in\Gamma$ and $k\in\mathbb{N}$,
\[
\psi_N(y) = \left(\frac{N}{2\pi}\right)^{\frac{n}{4}} s^N(y) \sum\limits_{0\le l\le k}N^{-l}b_l(y) + O\left(N^{-k-1}\right).
\]
Furthermore, $(\psi_N)$ is unique up to a family $(\Phi_N)\in C^{\infty}(I,H_N)$ satisfying
\[
\lnor \left(\frac{d}{dt}\right)^p \Phi_N(t) \rnor = O\left(N^{-k}\right)
\]
for all $p$ and $k$ uniformly on any compact of $I$. We call total symbol of$\psi_N$ the formal series $\sum N^{-l} a_l$ and principal symbol of $\psi_N$ the coefficient $a_0$.

Until the end of this section we will consider a time dependant Berezin-Toeplitz operator $T_{N,t} = \Pi_N(f(\cdot,t,N))$ such that the function $f$ has an expansion in $N$ like
\[
f(\cdot,\cdot,N)=f_0+N^{-1}f_1+\cdots
\]
with coefficients $f_l\in C^{\infty}(I\times M)$. We call $H_t=f_0(t,.)$ the principal symbol of $T_{N,t}$ and $H^{sub}_t=f_1(t,.)+\frac{1}{2}\Delta f_0(t,.)$ the sub-principal symbol.

\begin{proposition}[\cite{char03} Chapter 2.4]
\label{prop_Toeplitz_Lagr}
Let $(\psi_N)$ a Lagrangian state as in Definition \ref{def_Lagr_state} associated to $(\Gamma,s)$, then $(T_{N,t}\psi_N)$ is also a Lagrangian state associated to $(\Gamma,s)$. In addition, its principal symbol is $H_0|_{\Gamma}\cdot b_0$ where $b_0$ is the principal symbol of $(\psi_N)$.
\end{proposition}

Since $s(t,\cdot)$ is flat for all $t\in I$, there exists $\alpha\in C^{\infty}(\Gamma,\mathbb{C})$ such that $\nabla s=\alpha dt \otimes s$, where the covariant derivative is induced by the usual derivative over $\mathbb{C}_I$ and the connection of $L$. What's more, $s$ is unitary for all time, thus by differentiating $1 = h(s,s)$, with $h$ the natural metric over $\mathbb{C}_I \otimes L$, we get
\[
0 = 2 Re(h(\nabla s,s)) = 2 Re(\alpha) dt.
\]
Hence there exists $\tau\in C^{\infty}(\Gamma,\mathbb{R})$ such that $\nabla s = i\tau dt \otimes s$.

\begin{proposition}[\cite{char21} Chapter 2.2]
\label{prop_time_Lagr}
Let $(\psi_N)$ a Lagrangian state like in Definition \ref{def_Lagr_state} The family $((iN)^{-1}\partial_t\psi_N)$ is a Lagrangian state associated to $(\Gamma,s)$ with total symbol $(\tau+\hbar P)b$ where $P=\nabla_Z$ and $Z(t,x)\in T_{(t,x)}\Gamma$ is the projection of $\frac{\partial}{\partial t}$ on $T_{(t,x)}\Gamma$ parallel to  $T^{0,1}_x M$.
\end{proposition}

In the next steps, we will consider pre-quantum lifts of Hamiltonian flows, we recall the definition now.

\begin{definition}[\cite{char21} Chapter 1.1]
\label{defrelev}
Suppose $I$ is a real segment, consider $H_t \in C^1(I\times M,\mathbb{R})$ a function, $X_t$ its Hamiltonian vector field and $\phi_t$ the associated flow. We call prequantum lift of $\phi_t$ a time-parametrized group $\phi_t^L$ of diffeomorphisms of $L$, which preserves the connection and the metric, and such that for all time $t$
\[
\pi \circ \phi_t^L = \phi_t \circ \pi.
\]
\end{definition}

There is no unique prequantum lift as one expression gives a family of solutions by multiplying it by a complex number of modulus $1$. Here, we fix a convention.

\begin{proposition}[\cite{kost70}]
\label{th_relev}
We write $T_t^L(x) : L_x \rightarrow L_{\phi_t(x)}$ the parallel transport along $\phi_t$, then a prequantum lift of the Hamiltonian flow $\phi_t$ associated to $H_t$ is given by
\[
\phi_t^L(x) = e^{\frac{1}{i}\int_0^tH_r\circ\phi_r(x)dr}T_t^L(x).
\]
\end{proposition}
From now on we will call this expression ''the prequantum lift''. Let us suppose $(\Gamma,s)$ is obtained by propagating a Lagrangian submanifold $\Gamma_0$ of $M$ and a unitary flat section $s_0$ of $L \rightarrow \Gamma_0$ by a Hamiltonian flow. For a function $H_t$ let us consider its Hamiltonian flow $\phi_t$ and $\phi_t^L$ its pre-quantum lift defined in \ref{defrelev}. We denote,
\begin{align*}
\Gamma(t) & = \Gamma_t = \phi_t(\Gamma_0)\\
s(t,\phi_t (x)) & = s_t(\phi_t (x)) = \phi_t^L(x)(s_0(x))
\end{align*}
$s_t$ then stays flat and unitary for all time by theorem \ref{th_relev}. Let us write  $Y(t,x) =\frac{\partial}{\partial_t} + X_t(x)$ where $X_t$ is the Hamiltonian vector field of $H_t$. We can combine Proposition \ref{prop_Toeplitz_Lagr} and Proposition \ref{prop_time_Lagr} to describe the action on Lagrangian states of the operator $\left(\frac{1}{iN}\frac{\partial}{\partial_t} + T_{t,N}\right)$.

\begin{proposition}[\cite{char21} Chapter 2.2]
\label{thzeta}
Let $(\psi_N)$ a Lagrangian state associated to $(\Gamma,s)$ with total symbol $b= \sum N^{-l} b_l$. Then $\left(\frac{1}{iN}\frac{\partial \psi_N}{\partial_t} + T_{t,N}\psi_N\right)$ is a Lagrangian state associated to $(\Gamma,s)$ with total symbol $\frac{1}{N}\left(\frac{1}{i}\nabla_Y +\zeta\right)b_0 +O\left(\hbar^2\right)$ where $\zeta \in C^{\infty}(\Gamma)$.
\end{proposition}

This result gives direct information on the associated Schrödinger equation \eqref{eq_sch}.

\begin{corollaryp}[\cite{char21} Chapter 2.2]
Let $(\psi_{0,N})$ be a Lagrangian state associated to $(\Gamma_0,s_0)$, the solution to the Schrödinger equation
\begin{align}
\label{eq_sch}
\frac{1}{iN}\frac{\partial\psi_N}{\partial_t}+T_{N,t}\psi_N & = 0 \nonumber \\
\psi_N(0,\cdot) & = \psi_{0,N}
\end{align}
is a family of Lagrangian states associated to $(\Gamma,s)$ with principal symbol $b_0$ satisfying the transport equation $\frac{1}{i}\nabla_Y b_0 +\zeta b_0 = 0$.
\end{corollaryp}

In order to lighten the computations, we use the following notations from now on
\begin{align*}
\sigma^{preq}_t(x) & = \phi_t^L(x)(s_0(x)),\\
\sigma^{q}_t(x) & = T_t^L(x)(s_0(x)),
\end{align*}
in particular Theorem \ref{th_relev} gives that
\[
\sigma^{preq}_t(x) = e^{\frac{1}{i} \int_0^t H_r\circ\phi_r(x)dr} \sigma^{q}_t(x).
\]

\subsection{Quantum propagator}

We will now focus on solving equation \eqref{eq_sch} in order to write the quantum propagator as a Lagrangian state, which will be useful in the proof of Theorem \ref{th_intro}. Recall that, for an operator $T_N$, its quantum propagator $e^{-iNtT_N}$ is the operator such that for $\psi_N$ in its domain and all $t\in\Rm$

\[
\lpar \frac{1}{iN}\frac{\partial}{\partial_t} + T_N\rpar \lpar e^{-iNtT_N} \psi_N\rpar = 0.
\]
Furthermore, if $(e_j)_{j\in J}$ is a Hilbert basis of eigenfunctions of $T_N$ with eigenvalues $(\lambda_j)_{j\in J}$ then the Schwartz kernel of $e^{-iNtT_N}$ is

\[
\sum\limits_{j\in J} e^{-iNt\lambda_j}e_j(x)\overline{e_j(y)}.
\]

First, we will need some isomorphisms of line bundles. We write $K=\Lambda^{n,0}T^*M$ the canonical bundle of $M$ and $K_t = K|_{\Lambda_t}$ the restriction of $K$ to $\Gamma_t$. We also write $\det(T^*\Gamma_t)=\bigwedge^nT^*\Gamma_t$, and $\det(T^*\Gamma)=\bigwedge^{n+1}T^*\Gamma$ the determinant bundles. For $t\in I$, $K_t$ and $\det(T^*\Gamma_t)\otimes \mathbb{C}$ are isomorphic, the isomorphism is defined for every $x\in\Gamma_t$ by,
\begin{align}
\label{eqiso}
(K_t)_x & \rightarrow \det(T_x^*\Gamma_t)\otimes\mathbb{C} \nonumber\\
\Omega & \mapsto \Omega|_{T_x\Gamma_t}.
\end{align}
The injectivity comes from the fact that $\lpar T_x\Gamma_t\otimes \mathbb{C}\rpar \cap T_x^{0,1}M = \{0\}$, since $\Gamma_t$ is Lagrangian. Besides, $\Gamma_t$ is a fibre of $\Gamma\rightarrow I$, thus the differentials of the injection $\Gamma_t \rightarrow \Gamma$ and the projection $\Gamma\rightarrow\mathbb{R}$ give an exact sequence
\[
0\rightarrow T_x\Gamma_t \rightarrow T_{(t,x)}\Gamma \rightarrow \mathbb{R}=T_t^*I \rightarrow 0.
\]
Taking the associated determinant bundles, and using that $\mathbb{R}$ has a canonical volume element, we get for all $x\in\Gamma_t$ the isomorphism
\[
\det(T^*_x\Gamma_t) \simeq \det(T^*_{(t,x)}\Gamma),
\]
hence
\[
\det(T^*\Gamma_t) \simeq \det(T^*\Gamma)|_{\Gamma_t}.
\]

\begin{lemma}[\cite{char21} Chapter 2.3]
Combining these isomorphisms, we get
\begin{align*}
\Xi : K_{\Gamma}=(\mathbb{C}_I\boxtimes K)|_{\Gamma} & \simeq \det(T^*\Gamma)\otimes \mathbb{C}\\
(1\boxtimes \alpha)|_{\Gamma} & \mapsto j^*(dt \wedge \alpha)
\end{align*}
where $\alpha\in\Omega^{n,0}(M)$ and $j$ is the embedding $\Gamma \rightarrow I\times M$.
\end{lemma}

On one hand, $K_{\Gamma}$ has a natural connection induced by the Chern connection on $K$, which gives a derivation $\nabla_Y$ on the sections of $K_{\Gamma}$. On the other hand, the Lie derivative $\mathcal{L}_Y$ acts on the differential forms of $\Gamma$ so, in particular, on the sections of $\det(T^*\Gamma)$. We then give a link between these two derivations,
\[
\mathcal{L}_Y(\Xi\cdot) = \Xi(\nabla_Y+i\theta)
\]
where $\theta \in C^{\infty}(\Gamma)$, since $\mathcal{L}_Y$ and $\nabla_Y$ are both derivatives in the direction $Y$.

\begin{proposition}[\cite{char21} Chapter 2.3]
The function $\zeta$ from Theorem \ref{thzeta} is given by $\zeta = \frac{1}{2}\theta + H^{sub}|_{\Gamma}$, where $H^{sub}$ is the sub-principal symbol of the operator $T_{t,N}$.
\end{proposition}

Now, let us give an expression of the solution of equation \eqref{eq_sch}. For all $t\in I$, the differential of $\phi_t$ restricted to $T\Gamma_0$ gives an isomorphism with values in $T\Gamma_t$. Using the identification \eqref{eqiso} we lift $\phi_t$ into an isomorphism $K|_{\Gamma_0} \rightarrow K|_{\Gamma_t}$. More precisely, for all $x\in\Gamma_0$, $u\in K_x$ and $v\in \det(T_x\Gamma_0)$, we define $\mathcal{E}_t(x)u\in K_{\phi_t(x)}$ such that
\[
\lpar\mathcal{E}_t(x)u\rpar\lpar(d_x\phi_t)_*v\rpar = u(v).
\]
Then, the parallel transport can also be restricted to get an isomorphism $K|_{\Gamma_0}\rightarrow K|_{\Gamma_t}$, we thus define the complex number $C_t(x)$ such that $\mathcal{E}_t(x) = C_t(x)T_t^K(x)$. In particular, the function $(t,x) \mapsto C_t(x)$ is continuous, hence bounded on $I\times M$.

\begin{corollaryp}[\cite{char21} Chapter 2.3]
The solution $b\in C^{\infty}(\Gamma,L')$ to the transport equation $\frac{1}{i}\nabla_Yb+\zeta b=0$ is
\[
b(t,\phi_t(x)) = C_t(x)^{\frac{1}{2}}e^{-i\int_0^tH_r^{sub}\circ \phi_r(x)dr}T_t^{L'}(x)(b(0,x))
\]
with the square root of $C_t(x)$ chosen continuously and $C_0 = 1$.
\end{corollaryp}

\begin{proof}
Since the integral curves of $Y$ are the $t\mapsto (t,\phi_t(x))$, the solution is of the form
\[
b(t,\phi_t(x)) = e^{-i\int_0^t\zeta(r,\phi_r(x))dr} T_t^{L'}(x)(b(0,x)).
\]
It is enough to prove it for $H^{sub}|_{\Gamma}=0$, furthermore, if $\tilde{b}$ satisfies $\nabla_Y\tilde{b}=0$ then $b=f\tilde{b}$ solves the transport equation if and only if $-iY\cdot f + \zeta f=0$. Then, it is enough to prove that $f:(t,\phi_t(x)) \mapsto C_t(x)^{\frac{1}{2}}$ is a solution of this equation.

First the isomorphism $I\times \Gamma_0 \simeq \Gamma$, $(t,x) \rightarrow (t,\phi_t(x))$ sends the vector field $\partial_t$ to $Y$. Then, the solutions of $\mathcal{L}_{\partial_t}\beta=0$ with $\beta\in\Omega^{n+1}(I\times\Gamma_0)$ have the form $\beta = dt\wedge\beta_0$ with $\beta_0\in\Omega^n(\Gamma_0)$. Hence the solutions to $\mathcal{L}_Y\alpha = 0$ with $\alpha\in\Omega^{n+1}(\Gamma)$ are of the form
\[
\alpha(t,\phi_t(x)) = dt\wedge(\phi_t^{-1})^*\alpha|_{t=0}(x)
\]
Using the isomorphism $\Xi : K_{\Gamma} \simeq \det(T^*\Gamma)\otimes \mathbb{C}$ and the application $\mathcal{E}_t(x)$ we get,
\[
\Xi^{-1}\alpha(t,\phi_t(x)) = \mathcal{E}_t(x)(\Xi^{-1}\alpha(0,x))
\]
Furthermore, the solutions $\alpha'\in C^{\infty}(\Gamma,K_{\Gamma})$ of $\nabla_Y\alpha'=0$ are given by
\[
\alpha'(t,\phi_t(x)) = T_t^K(x)(\alpha'(0,x)).
\]
If at time $0$, $\alpha'(0,x) = \Xi^{-1}\alpha(0,x)$, then $C\alpha'=\Xi^{-1}\alpha$ with $C\in C^{\infty}(\Gamma)$ defined by $C(t,\phi_t(x)) = C_t(x)$. Thus
\begin{align*}
0 & = \mathcal{L}_Y(\alpha) = \mathcal{L}_Y(C\Xi\alpha') = (Y\cdot C)\Xi\alpha' + C\mathcal{L}_Y(\Xi\alpha')\\
 & = (Y\cdot C)\Xi\alpha' + C\Xi\nabla_Y\alpha' + 2i\zeta C \Xi\alpha'
\end{align*}
with the hypothesis $\nabla_Y\alpha'=0$, so $Y\cdot C + 2i\zeta C =0$ and then $-iY\cdot C^{\frac{1}{2}} + \zeta C^{\frac{1}{2}} =0$.
\end{proof}

We finally get the wanted result as a corollary.

\begin{corollaryp}[\cite{char21} Chapter 4]
\label{corpropq}
Let $(T_{N,t})$ be a smooth family of Berezin-Toeplitz operators with real principal symbol $H_t$, and sub-principal symbol $H^{sub}_t$. Then the Schwartz kernel of the quantum propagator of $(T_{N,t})$ times $\left(\frac{N}{2\pi}\right)^{-\frac{n}{2}}$, is a family of Lagrangian states associated to $(\Gamma,s)$ with principal symbol $\sigma$ where
\begin{align*}
\Gamma & = \lacc (t,\phi_t(x),x)/\; t\in I,\; x\in M\racc,\\
s(t,\phi_t(x),x) & = \sigma^{preq}_t(x) ,\\
\sigma (t,\phi_t(x),x) & = C_t(x)^{\frac{1}{2}} e^{-i\int_0^t H^{sub}_r \circ\phi_r(x)dr} T_t^{L'}(x).
\end{align*}
\end{corollaryp}

\section{Concentration estimates}
\label{sec_Lp}

This section consists of three main results; first, the main goal of this paper, $L^p$ upper bounds of quasimodes of Toeplitz operators on compact Kähler manifolds. Then, the sharpness of these estimates through examples in any dimension. In a third time, we use the symbolic calculus of Toeplitz operators on $\Cm^n$ in order to write their quasimodes as quasimodes of Toeplitz operators on the Torus, under some hypothesis. Doing so, we will prove Theorem \ref{th_intro} on flat spaces. We will also prove a similar result on the FBI transform of pseudodifferential operators' quasimodes.

\subsection{Quasimodes of Toeplitz operators}

In this section, we will use the Lagrangian state representation of quantum propagators to extract eigenfunctions and bound their $L^{\infty}$ norm. The $L^p$ estimates will be obtained for all $2\le p\le +\infty$ next by interpolation. In all this subsection $T_N$ is a self-adjoint, time independent, Toeplitz operator, with principal symbol $H$ and sub-principal symbol $H^{sub}$. Recall that $I$ is a fixed real segment.

In order to isolate eigenfunctions, we want to consider operators of the form $\rho(N(E-T_N))$, which are related to the quantum propagator through time Fourier transform. More precisely, we write the $N$-Fourier transform
\begin{align*}
\mathcal{F}_N(\rho)(t) & = \left(\frac{N}{2\pi}\right)^{\frac{1}{2}}\int_{\mathbb{R}}e^{-iNtE}\rho(E)dE,\\
\mathcal{F}_N^{-1}(p)(E) & = \left(\frac{N}{2\pi}\right)^{\frac{1}{2}}\int_{\mathbb{R}}e^{iNtE}p(t)dt,
\end{align*}
which transforms the time variable into the energy variable. We shall write $\hat{\rho} = \mathcal{F}_1(\rho)$ the usual Fourier transform.

\begin{lemma}
Let $\rho\in C^{\infty}(\mathbb{R})$ whose Fourier transform is $C^{\infty}$ with compact support in $I$. Let $E$ be a regular value of $H$, then
\[
\rho(N(E-T_N)) = N^{-\frac{1}{2}}\mathcal{F}_N^{-1}(\hat{\rho}(t)U_{N,t})(E)
\]
\end{lemma}

\begin{proof}
For a fixed $N\in\mathbb{N}$, $T_N$ is self-adjoint, so there is a basis of eigenfunctions $(e_{N,j})_{j\in J}$ with eigenvalues $(\lambda_{N,j})_{j\in J}$. We then are able to write the Schwartz kernels as
\begin{align*}
\rho(N(E-T_N))(y,x) & = \sum\limits_{j\in J} \rho(N(E-\lambda_{N,j}))\overline{e_{N,j}(y)}e_{N,j}(x),\\
U_{N,t} & = \sum\limits_{j\in J}e^{-iNt\lambda_{N,j}}\overline{e_{N,j}(y)}e_{N,j}(x),\\
N^{-\frac{1}{2}}\mathcal{F}_N^{-1}(\hat{\rho}(t)U_{N,t})(E) & = \sum\limits_{j\in J}N^{-\frac{1}{2}}\mathcal{F}_N^{-1}(\hat{\rho}(t)e^{-iNt\lambda_{N,j}})(E)\overline{e_{N,j}(y)}e_{N,j}(x).
\end{align*}
But
\[
N^{-\frac{1}{2}}\mathcal{F}_N^{-1}(\hat{\rho}(t)e^{-iNt\lambda_{N,j}})(E) = \mathcal{F}_1^{-1}(\hat{\rho})(N(E-\lambda_{N,j})) = \rho(N(E-\lambda_{N,j}))
\]
hence the equality of the Schwartz kernels, and of the operators.
\end{proof}

In order to compute the Fourier transform of a Lagrangian state at a specific point $E$, we need another type of such states, which will depend on an immersed manifold depending on $E$. In our case, the immersed manifold we describe here was introduced by Charles and Le Floch \cite{char21}.

Consider a Lagrangian immersion $j: \Gamma \rightarrow M$, a flat unitary section $s$ of $j^*L$ and a formal series $\sum h^l b_l$ with coefficients $b_l \in C^{\infty}(j^*L')$. To begin with, if $y\in \Gamma$ we define a piece of Lagrangian state at $j(y)$. We suppose there exists an open set $V$ of $M$ such that $j:\Gamma\rightarrow V$ is a proper embedding, so that $j(\Gamma)$ is a closed submanifold of $V$. Then there exists sections $F:V\rightarrow L$ and $a_l:V\rightarrow L'$ with $\overline{\partial}F$ and $\overline{\partial}a_l$ vanishing at any order along $j(\Gamma)$, $j^*F=s$ and $j^*a_l = b_l$ and $|F|<1$ on $V\backslash j(\Gamma)$. We then have a local Lagrangian section, for $x\in V$
\[
\left(\frac{N}{2\pi}\right)^{\frac{n}{4}}F^N(x)\sum\limits_{0\le l\le A}N^{-l}a_l(x).
\]
This section is not unique, but we can prove that if $F'$, $a_l'$ satisfy the same conditions then
\[
F^N\sum\limits_{0\le l\le A}N^{-l}a_l = (F')^N\sum\limits_{0\le l\le A}N^{-l}a_l' + O\left(N^{-A-1}\right)
\]
where the error is uniform on every compact of $V$. In the general case, for all $y\in \Gamma$, the theorem of normal forms of immersions gives neighbourhoods $U$ and $V$ of $y$ and $j(y)$ such that $j:U\rightarrow V$ is a proper embedding. We then construct a local Lagrangian state on $V$ as before.

Now let us suppose there is a compact $K$ of $\Gamma$ such that for all $l$, $b_l$ has compact support in $K$. Let $V$ be an open set of $M$ and $(U_i)_{i\in I}$ a family of open sets of $\Gamma$ such that for all $i\in I$, $j:U_i\rightarrow V$ is a proper embedding, and $K\cap j^{-1}(V) \subset \bigcup_{j\in I}U_i$. For each $i\in I$ we can do the same construction as before with sections $F_i$ and $a_{l,i}$ defined on $V$ which extend $s|_{U_i}$ and $b_l|_{U_i}$ respectively, and we write
\[
\psi^i_{A,N} = \left(\frac{N}{2\pi}\right)^{\frac{n}{4}}F_i^N(x)\sum\limits_{0\le l\le A}N^{-l}a_{l,i}(x).
\]

\begin{lemma}[\cite{char21} Chapter 5.2]
\label{lLagr}
There exists a family $(\psi_N \in \mathcal{H}_N)$ such that for all $x\in M$ and all $A\in\mathbb{N}$
\begin{itemize}
\item if $x\notin j(K)$, $|\psi_N| = O\left(N^{-A}\right)$ on a neighbourhood of $x$.
\item If $j^{-1}(x)\cap K = \lacc y_i/\; i\in J\racc$, then $\psi_N = \sum\limits_{i\in J} \psi_{A,N}^i +O\left(N^{-A}\right)$ on a neighbourhood of $x$, where the $\psi^i_{A,N}$ are built like before from open sets $V$ and $(U_i)_{i\in J}$ such that $V$ is a neighbourhood of $x$ and for each $i\in J$, $U_i$ is a neighbourhood of $y_i$.
\end{itemize}
Furthermore, this family is unique modulo $O\left(N^{-\infty}\right)$.
\end{lemma}

We call this family a Lagrangian state associated with the triple $\lpar j,s,\sum\hbar^l b_l\rpar$ where $j: \Gamma\rightarrow M$ is an immersion, $s$ is a flat and unitary section of $j^*L$ and $\sum\hbar^l b_l$ is a formal series with coefficients in $C^{\infty}(N,j^*L')$. We still call $\sum\hbar^l b_l$ the total symbol and $b_0$ the principal symbol of the Lagrangian state. These definitions are a bit different from the Lagrangian states of Definition \ref{def_Lagr_state}, in particular we cannot get the total symbol by evaluating the state on $j(K)$ because of possible multiple points.

Let us remind that $\nabla s = i\tau dt \otimes s$ where $\tau\in C^{\infty}(\Gamma,\mathbb{R})$.
\begin{proposition}[\cite{char21} Chapter 5.3]
\label{prop_Fourier_Lagr}
Let $(\psi_N \in C^{\infty}(I,H_N))$ be a family of Lagrangian states associated to $(\Gamma,s)$, with supports in a compact set $I$ independent of $N$, and let $-E$ be a regular value of $\tau$. Then, $\mathcal{F}_N^{-1}(\psi_N)(E)$ is a Lagrangian state associated with the Lagrangian immersion

\begin{align*}
j_E:\Gamma_E=\tau^{-1}(-E) & \rightarrow M\\
(t,x) & \mapsto x
\end{align*}
to the section $s^E:(t,x) \mapsto e^{itE}s(t,x)$ and with principal symbol
\begin{align*}
\sigma^E(t,x) = B(t,x)^{-\frac{1}{2}}\sigma(t,x).
\end{align*}
Here, $\sigma$ is the principal symbol of $(\psi_N)$ and $B$ is such that $d\tau \wedge \alpha = iB(t,x)dt\wedge \alpha$ for all $\alpha \in K_x$, the square root chosen with positive real part.
\end{proposition}

We notice that the norm of this Lagrangian state will be equal to $|B(t,x)^{-\frac{1}{2}}|$ times the norm of $\psi_N$, but one can check that $B(t,x)$ is non-zero and continuous, hence this term is bounded.

We can write the quantum propagator as a Lagrangian state, and we saw that
\[
\rho(N(E-T_N)) = N^{-\frac{1}{2}}\mathcal{F}_N^{-1}(\hat{\rho}(t)U_{N,t})(E)
\]
hence, applying Proposition \ref{prop_Fourier_Lagr} to this term we get the next result.

\begin{proposition}[\cite{char21} Chapter 6.2]
\label{prop_quantum_prop}
Let $E$ be a regular value of $H$ and $\rho$ be a $C^{\infty}$ function which Fourier transform is $C^{\infty}$ with compact support. Then the Schwartz kernel of $\rho(N(E-T_N))$ is a Lagrangian state associated with the triple $\lpar j^E,s^E,\sigma^E\rpar$ where

\begin{align*}
\Gamma_E & = \mathbb{R}\times H^{-1}(E)\\
j_E(t,x) & = (\phi_t(x),x)\in M^2\\
s^E\lpar j_E(t,x)\rpar & = \sigma^{q}_t(x)\\
\sigma^E\lpar j_E(t,x)\rpar & = \hat{\rho}(t)\left(\frac{C_t(x)}{B(t,x)}\right)^{\frac{1}{2}}e^{-i\int_0^t H^{sub}\circ\phi_r(x)dr}T_t^{L'}(x)
\end{align*}
for $(t,x)\in \Gamma_E$.
\end{proposition}

It is now possible to give an explicit formula for $\rho(N(E-T_k))$.

\begin{theorem}
\label{corprincipal}
Let $E$ be a regular value of $H$ and $\rho$ be a $C^{\infty}$ function which Fourier transform is $C^{\infty}$ with compact support. We denote $\phi_t$ the Hamiltonian flow of $H_t$. There exists $\epsilon>0$ such that for all $x\in H^{-1}(E)$ the period $T_x\in ]0,+\infty]$ of $\lpar t\mapsto \phi_t(x)\rpar$ is greater than $2\epsilon$. Then for all $(y,x)\in M^2$
\begin{align}
\label{eq_result}
& \rho(N(E-T_N))(y,x) = \frac{N^{n-\frac{1}{2}}}{(2\pi)^n}\nonumber\\
& \times \sum\limits_{t\in S_{\epsilon}(y,x)} e^{iNtE-N|E-H(x)|}s^N(t,y,x)\hat{\rho}(t)B(t,x)^{-\frac{1}{2}}\sigma(t,y,x) + O(1)
\end{align}
where $s$ and $\sigma$ are given in Corollary \ref{corpropq} and
\[
S_{\epsilon}(y,x) = \lacc t_{\epsilon}(k,x,y) /\; k\in\mathbb{Z},[k\epsilon,(k+1)\epsilon]\cap I \neq \emptyset\racc,
\]
where $t_{\epsilon}(k,x,y)$ is defined with $I_{\epsilon}(k,x,y)$ the set of minimizer of $|y-\phi_t(x)|$ on $[k\epsilon,(k+1)\epsilon]\cap I$,
\[
t_{\epsilon}(k,x,y) = \begin{cases}
k\epsilon & \text{, if } I_{\epsilon} = \{(k+1)\epsilon\}\\
\inf I_{\epsilon} & \text{, else.} 
\end{cases}
\]

With this definition, if $k\neq k'$ then $t_{\epsilon}(k,x,y) \neq t_{\epsilon}(k',x,y)$ for all $x,y$. Furthermore, the number of elements in $S_{\epsilon}(y,x)$ is finite, and it is bounded independently of $x$ and $y$, as $I$ is a segment.
\end{theorem}

\begin{proof}
By definition of a Lagrangian state on an immersed manifold, we just have to check that the formula \eqref{eq_result} defines a function in $H^N$ modulo $O(1)$ which satisfy the asymptotic equalities of Proposition \ref{lLagr}.

Since $\hat{\rho}$ is continuous and has compact support, and thanks to the boundedness of $C_t(x)$ and $B(t,x)^{-\frac{1}{2}}$, the term $e^{iNtE}\hat{\rho}(t)B(t,x)^{-\frac{1}{2}}\sigma(t,y,x)$ is uniformly bounded thus, we just have to check the asymptotic expansion of the term $e^{-N|E-H(x)|}s^N(t,y,x)$ in equation \eqref{eq_result}. Here, the compact $K$ is replaced by $I\times H^{-1}(E)$, so let $(y,x) \notin \lacc (\phi_t(x),x)/t\in I, x\in H^{-1}(E)\racc$,
\begin{itemize}
\item if $H(x) \neq E$, then $|e^{-N|E-H(x)|}s^N(t,y,x)| \le e^{-CN} = O\left(N^{-\infty}\right)$.
\item If there is no $t$ such that $y = \phi_t(x)$, then according to corollary \ref{corpropq} $|s^N(t,y,x)|$ is strictly smaller than $1$ so
\[
\lver e^{-N|E-H(x)|}s^N(t,y,x) \rver \le |s(t,y,x)|^N = O\left(N^{-\infty}\right).
\]
\end{itemize}
Now let us suppose $H(x)=E$ and there exists $t^*\in I$ such that $y=\phi_{t^*}(x)$, then writing $T\in ]0,+\infty]$ the period of $t\mapsto\phi_t(x)$, the set $S_{\epsilon}$ is equal to
\[
S_{\epsilon}(x,y) = \lacc t^*+kT\in I /\; k\in\mathbb{Z}\racc\cup\{t_1,\cdots,t_r\}
\]
where the $(t_j)_{1\le j\le r}$ doesn't satisfy $\phi_{t_j}(x)=y$. Then, according to the previous arguments, the corresponding terms are $O\left(N^{-\infty})\right)$, while the others can be computed with Corollary \ref{corpropq}, which gives for $\rho(N(E-T_N))(y,x)$
\[
\frac{N^{n-\frac{1}{2}}}{(2\pi)^n} \sum\limits_{t=t^*+kT\in I} e^{iNtE}\sigma^{preq}_t(x)^N \hat{\rho}(t)B(t,x)^{-\frac{1}{2}}C_t(x)^{\frac{1}{2}}e^{-i\int_0^tH_r^{sub}\circ\phi_r(x)}T_t^{L'}(x) + O(1)
\]
But we chose $H$ time independent, so $H\circ\phi_t(x)=H(x)$ for all $t\in I$, hence
\[
e^{iNtE}\sigma^{preq}_t(x)^N = \lpar e^{itE-i\int_0^tH(x)dr}\sigma^{q}_t(x) \rpar^N = \sigma^{q}_t(x)^N
\]
We finally get
\begin{align*}
& \rho(N(E-T_N))(y,x) = \frac{N^{n-\frac{1}{2}}}{(2\pi)^n}\\
& \times \sum\limits_{t\in I/y=\phi_t(x)} \sigma^{q}_t(x)^N \hat{\rho}(t)\left(\frac{C_t(x)}{B(t,x)}\right)^{\frac{1}{2}}e^{-i\int_0^tH_r^{sub}\circ\phi_r(x)}T_t^{L'}(x) + O(1)
\end{align*}
which is consistent with Proposition \ref{prop_quantum_prop}.

It only remains to prove the holomorphy, but on $J_E(\Gamma_E)$ the formula match Proposition \ref{prop_quantum_prop} modulo $O(1)$, while outside that set, it is equal to $O\lpar N^{-\infty}\rpar$. Hence, adding some negligible terms won't change the formula modulo $O(1)$, and we can make it holomorphic everywhere.

In the end, the formula match the Lagrangian state described in Proposition \ref{prop_quantum_prop}, so there is equality modulo $O(1)$.
\end{proof}

Now we want to use this formula to estimate $\rho(N(E-T_N))(x,x)$, as it is equal to a weighted sum of squares of $T_N$'s eigenfunctions' norms evaluated in $x$. For that purpose, we first need a preliminary lemma on the distance between eigenvalues.

\begin{lemma}
\label{lem_dist_eigen}
Let $E$ be a regular value of $H$, then there exists $C>0$ such that for $N$ large enough, the interval $\left[E-\frac{C}{N},E+\frac{C}{N}\right]$ contains at least one eigenvalue of $T_N$.
\end{lemma}

\begin{proof}
Let $\rho$ be a $C^{\infty}$ function which Fourier transform is $C^{\infty}$ with support inside $]-\epsilon,\epsilon[$, where $\epsilon$ is given by corollary \ref{corprincipal}. Hence $\supp (\rho) \bigcap S_{\epsilon}(y,x) = \{0,t_n\}$ where $t_n$ corresponds to a negligible term, so we can write
\[
\rho(N(E-T_N))(x,x) = \frac{N^{n-\frac{1}{2}}}{(2\pi)^n} e^{-N|E-H(x)|}s^N(0,x,x)\hat{\rho}(0)B(0,x)^{-\frac{1}{2}}\sigma(0,x,x) + O(1).
\]
Since $C_t(x)$ and $B(t,x)^{-\frac{1}{2}}$ are uniformly bounded, this expression is integrable, and there exists $\omega >0$ such that
\[
\sum\limits_{j\in J} \rho(N(E-\lambda_{N,j})) = \int_{M} \rho(N(E-T_N))(x,x)dx = \hat{\rho}(0)\omega N^{n-\frac{1}{2}} + O(1)
\]
where $(\lambda_{N,j})_{j\in J}$ are the eigenvalues of $T_N$. We write this expression as $\gamma_{E,N}(\rho)$. Now, let us check that it satisfy the hypothesis of the Tauberian lemma (see Chapter 6 of \cite{brum95}). $\omega$ might depend on $E$, but it is uniformly bounded in $N$ and for $E$ in a compact set. Consider a number $C_1$, if it is a regular value of $H$ then $\gamma_{C_1,N}(\rho) = O(N^{n-\frac{1}{2}})$ else $\gamma_{C_1,N}(\rho) = O(1) = O(N^{n-\frac{1}{2}})$. So, choosing a $\rho$ which is greater than $1$ on $\left]-\infty,-\left(n-\frac{1}{2}\right)\right[$ we get
\[
\#\lacc j/\; \lambda_{N,j}\le C_1-\frac{n-\frac{1}{2}}{N}\racc = O(N^{n-\frac{1}{2}}).
\]
We can apply the Tauberian lemma, there exists a $C>0$ such that
\[
\#\lacc j\in J/\; |\lambda_{N,j}-E|\le \frac{C}{N}\racc = \sqrt{\frac{2}{\pi}}C\omega N^{n-\frac{1}{2}} + \mathit{o}(N^{n-\frac{1}{2}})
\]
and this number is greater than $0$ for $N$ large enough.
\end{proof}

Using this lemma, we can finally prove the main theorem. Recall that $T_N$ is a time independent, self-adjoint Berezin-Toeplitz on the compact Kähler manifold $M$ of dimension $n$, with principal symbol $H$.

\begin{theorem}
\label{th_mine}
Let $E$ be a regular value of $H$, $\mu_N$ a sequence of real numbers such that $|\mu_N - E| \le \frac{C}{2N}$ with $C>0$ and for $N$ large enough. Let $V_N$ be an associated $L^2$-normalised quasimodes, that is
\[
T_N V_N = \mu_N V_N + O_{L^2(M)}(N^{-\infty}).
\]
For all $N\in\mathbb{N}$ and $p\in[2,+\infty]$, we have
\[
\|V_N\|_{L^p(M)} = O\left(N^{\left(n-\frac{1}{2}\right)\left(\frac{1}{2}-\frac{1}{p}\right)}\right)
\]
\end{theorem}

\begin{proof}
We first take $p=+\infty$, and we will deduce the others from this one. Since $T_N$ is self-adjoint, there is a Hilbert basis $(e_{N,j})_{j\in J}$ of eigenfunctions, we then write $(\lambda_{N,j})_{j\in J}$ the associated eigenvalues, with possible repetition. In particular there exists scalars $(a_{N,j})_{j\in J}$ such that

\[
V_N = \sum\limits_{j\in J} a_{N,j} e_{N,j}
\]
with $\sum\limits_{j\in J} |a_{N,j}|^2 = 1$. Let $J_N = \lacc j\in J /\; \lver \lambda_{N,j}-\mu_N\rver \le \frac{C}{2N} \racc$. By hypothesis on $V_N$

\[
\sum\limits_{j\in J} |a_{N,j}|^2 |\lambda_{N,j}-\mu_N|^2 = \lnor T_N V_N - \mu_N V_N \rnor_{L^2(M)} = O(N^{-\infty}).
\]
In particular

\[
\frac{C}{2N} \sum\limits_{j\notin J_N} |a_{N,j}|^2 \le \sum\limits_{j\notin J_N} |a_{N,j}|^2 |\lambda_{N,j}-\mu_N|^2 = O(N^{-\infty}),
\]
so $\sum\limits_{j\notin J_N} |a_{N,j}|^2 = O(N^{-\infty})$. Now, if $\rho\in C^{\infty}(\mathbb{R})$
\[
\rho(N(E-T_N))(y,x) = \sum\limits_{j\in J}\rho(N(E-\lambda_{N,j}))\overline{e_{N,j}(y)}e_{N,j}(x).
\]
Let $\rho$ be such that
\begin{itemize}
\item $\rho\ge \chi_{[-C,C]}$,
\item $\hat{\rho}$ is $C^{\infty}$ with compact support in $I$,
\end{itemize}
We can take for example, $2\chi_{[-2C,2C]}*(\mathcal{F}_N^{-1}(\ell))^2$ where $\ell\in C^{\infty}_c(I)$ is such that
\[
(\mathcal{F}_N^{-1}(\ell))^2 \ge \frac{1}{2}
\] 
on $[-C,C]$. According to Lemma \ref{lem_dist_eigen}, there exists $k\in J_N$ such that for $N$ large enough $|E-\lambda_{N,k}|\le \frac{C}{2N}$, and by construction of $J_N$, it implies that $|E-\lambda_{N,j}|\le \frac{C}{N}$ for all $j\in J_N$, that is $J_N \subset \lacc j\in J /\; N|E-\lambda_{N,j}|\le C \racc$. Using this inclusion and $\sum\limits_{j\notin J_N} |a_{N,j}|^2 = O(N^{-\infty})$, we get for all $x\in M$ 

\begin{align*}
|V_N(x)|^2
		& \le \sum\limits_{j\in J_N} |a_{N,j}|^2 |e_{N,j}(x)|^2 + O(N^{-\infty})\\
		& \le \chi_{[-C,C]}(N(E-T_N))(x,x) + O(N^{-\infty})\\
		& \le \rho(N(E-T_N))(x,x) + O(N^{-\infty}).
\end{align*}
According to Theorem \ref{corprincipal}, the term $\frac{1}{N^{n-\frac{1}{2}}}\rho	(N(E-T_N))(x,x)$ is equal to a uniformly bounded function on $M$ plus a negligible term. By taking the $\sup$ on $M$ in the previous inequality, we get
\[
\|V_N\|_{L^{\infty}(M)}^2 = O\left(N^{n-\frac{1}{2}}\right).
\]
Then by interpolation, since the functions are $L^2$-normalised, we get
\[
\|V_N\|_{L^p(M)}^p \le \|V_N\|_{L^{\infty}(M)}^{p-2} \|V_N\|_{L^2(M)}^2 = O\left(N^{\left(n-\frac{1}{2}\right)\left(\frac{p}{2}-1\right)}\right).
\]
\end{proof}

In fact, this result also contains the information of where the quasimodes concentrate. Using the equation in the proof of Lemma \ref{lem_dist_eigen} we get that for any compact $K$ such that $K\cap H^{-1} (E) = \emptyset$ and for any $2<p\le\infty$
\[
\lnor V_N \rnor_{L^p(K)} = O(N^{-\infty}).
\]

\subsection{Sharpness of the estimates}

In order to prove the sharpness of Theorem \ref{th_mine}, we show that the bound is in fact an asymptotic equivalent for well-chosen eigenfunctions of specific operators on the projective space.

We consider the projective space $\mathbb{CP}^n$ of dimension $n$, for $(z_0,\cdots,z_n)\in\mathbb{C}^{n+1}\backslash \{0\}$ we write $[z_0,\cdots,z_n]$ the associated line in $\mathbb{CP}^n$, then for $r\in\{0,\cdots,n\}$ we define the chart $(U_r,\Phi_r)$ by
\begin{align*}
\Phi_r : U_r & \rightarrow \mathbb{C}^n\\
		\left[z_0,\cdots,z_n\right] & \mapsto \left(\frac{z_0}{z_r},,\cdots,\frac{z_{r-1}}{z_r},\frac{z_{r+1}}{z_r},\cdots,\frac{z_n}{z_r}\right)\\
\Phi_r^{-1} : \left(w_1,\cdots,w_n\right) & \mapsto \left[z_1,\cdots,z_{r},1,z_{r+1},\cdots,z_n\right]
\end{align*}
where $U_r = \lacc \left[z_0,\cdots,z_n\right]\in\mathbb{CP}^n/\; z_r \neq 0\racc$. Henceforth, we will use the local coordinates $(w_1,\cdots,w_n)$ to work on $\Cm P^n$. In this system of coordinates, the Fubini-Study symplectic form is given by $\omega_{[z_0,\cdots,z_n]} =$
\[
 \frac{i}{(1+\sum_{1\le j\le n}|w_j|^2)^2}\sum\limits_{1\le l,m\le n}\left(\left(1+\sum_{1\le j\le n}|w_j|^2\right)\delta_{l,m}-\overline{w_l}w_m\right)dw_l\wedge d\overline{w_m}.
\]
We then consider the Toeplitz quantization on this manifold as described by Le Floch (Chapter 4.4 of \cite{lef18}). For the hermitian complex line bundle we choose the dual of the tautological bundle

\[
\mathcal{O}(1) = \lacc \left([u],v\right)\in \mathbb{CP}^n\times \mathbb{C}^{n+1}/\; v\in \mathbb{C}\frac{\overline{u}}{|u|^2}\racc.
\]
We fix a reference holomorphic section
\begin{align}
s: U_r & \rightarrow \mathbb{C}^{n+1}\\
	[z_0,\cdots,z_n] & \mapsto \frac{1}{1+\sum\limits_{1\le j\le n}|w_j|^2}\left(\overline{\frac{z_0}{z_r}},\cdots,\overline{\frac{z_{r-1}}{z_r}},1,\overline{\frac{z_{r+1}}{z_r}},\cdots,\overline{\frac{z_n}{z_r}} \right).
\end{align}
then we can check, using Theorem \ref{thconstr} that $\mathcal{O}(1)$ is a pre-quantized bundle for $\Cm P^n$

\begin{align*}
\log(h(s,s)) & = -\log\left(1+\sum\limits_{1\le j\le n}|w_j|^2\right)\\
\overline{\partial}\partial(\log(h(s,s)) & = \overline{\partial}\left(\sum\limits_{1\le l\le n}\frac{-\overline{w_l}}{1+\sum_{1\le j\le n}|w_j|^2}dw_l\right)\\
		& = \sum\limits_{1\le l,m\le n}\frac{-\delta_{l,m}\left(1+\sum_{1\le j\le n}|w_j|^2\right)+\overline{w_l}w_m}{\left(1+\sum_{1\le j\le n}|w_j|^2\right)^2}d\overline{w_m}\wedge dw_l\\
		& = -i\omega.
\end{align*}
The space $\mathcal{H}_N$ is made of sections of the form
\[
[z_0,\cdots,z_n] \mapsto f(z_0,\cdots,z_n)\left(\frac{\overline{z}}{|z|^2}\right)^{\otimes N}
\]
where $|z|^2=\sum\limits_{0\le j\le n} |z_j|^2$ and with $f$ holomorphic and homogeneous of degree $N$ for that expression to be well-defined. The only functions satisfying such properties are the homogeneous polynomials of degree $N$, we write $\mathbb{C}_N[X_0,\cdots,X_n]$ the space of these functions. In particular, in local coordinates on $U_r$ we get
\[
[z_0,\cdots,z_n] \mapsto P(w_1,\cdots,w_{r-1},1,w_r,\cdots,w_n)\left(\frac{(\overline{w_1},\cdots,\overline{w_{r-1}},1,\overline{w_r},\cdots,\overline{w_n})}{1+|w|^2}\right)^{\otimes N}
\]
with $P\in\mathbb{C}_N[X_0,\cdots,X_n]$. Though, there exists a bijection between this space and the space $\mathbb{C}_{\le N}[Z_1,\cdots,Z_n]$ of polynomials of degree at most $N$ on $n$ variables, given by
\begin{align*}
P\in \mathbb{C}_N[X_0,\cdots,X_n] & \mapsto Q(Z_1,\cdots,Z_n) = P(Z_1,\cdots,Z_{r-1},1,Z_r,\cdots,Z_n)\\
Q\in \mathbb{C}_{\le N}[Z_1,\cdots,Z_n] & \mapsto P(X_0,\cdots,X_n) = X_r^N Q(\frac{X_0}{X_r},\cdots,\frac{X_{r-1}}{X_r},\frac{X_{r+1}}{X_r},\cdots,\frac{X_n}{X_r})
\end{align*}
The quantum space is thus described locally by
\[
\mathcal{H}_N = \lacc \frac{Q(w_1,\cdots,w_n)\tilde{w}^{\otimes N}}{\langle w \rangle^{2N}}/\; Q\in\mathbb{C}_{\le N}[X_1,\cdots,X_n]\racc.
\]
where $\tilde{w}=(\overline{w_1},\cdots,\overline{w_{r-1}},1,\overline{w_r},\cdots,\overline{w_n})$ and $\langle w \rangle=\sqrt{1+|w|^2}$. let us write $dw^{\wedge} = dw_1\wedge d\overline{w_1}\wedge\cdots\wedge dw_n\wedge d\overline{w_n}$, then the volume form on that space is
\begin{align*}
\omega^n & = \frac{i^n}{\langle w \rangle^{4n}} \sum\limits_{a,b\in\{1,\cdots,n\}^n} \bigwedge\limits_{1\le j\le n} (\langle w \rangle^2 \delta_{a_j,b_j} - \overline{w_{a_j}}w_{b_j})dw_{a_j}\wedge\overline{w_{b_j}}\\
		& = \frac{i^n}{\langle w \rangle^{4n}} \sum\limits_{\sigma,\tau\in\mathfrak{S}_n} \bigwedge\limits_{1\le j\le n} (\langle w \rangle^2 \delta_{\sigma(j),\tau(j)} - \overline{w_{\sigma(j)}}w_{\tau(j)})dw_{\sigma(j)}\wedge\overline{w_{\tau(j)}}\\
		& = \frac{i^n}{\langle w \rangle^{4n}} \sum\limits_{\sigma,\tau\in\mathfrak{S}_n} \left(\prod\limits_{1\le j\le n} (\langle w \rangle^2 \delta_{\sigma(j),\tau(j)} - \overline{w_{\sigma(j)}}w_{\tau(j)})\right)(-1)^{\epsilon(\sigma)+\epsilon(\tau)} dw^{\wedge}\\
		& = \frac{i^n}{\langle w \rangle^{4n}} \sum\limits_{\sigma,\tau\in\mathfrak{S}_n} \left(\prod\limits_{1\le j\le n} (\langle w \rangle^2 \delta_{j,\tau\sigma^{-1}(j)} - \overline{w_j}w_{\tau\sigma^{-1}(j)})\right)(-1)^{\epsilon(\tau\sigma^{-1})}dw^{\wedge}\\
		& = \frac{i^n n!}{\langle w \rangle^{4n}} \det(\langle w \rangle^2 \delta_{i,j} - \overline{w_i}w_{j}) dw^{\wedge}.
\end{align*}
The matrix in the $\det$ can be written $\langle w \rangle^2I_n-(\overline{w_i}w_j)$, where the second matrix has rank $1$ with $\langle w \rangle^2$ as only non-zero eigenvalue and $\overline{w}$ for eigenvector. Hence, the determinant is equal to $\langle w \rangle^{2(n-1)}$ and
\[
\frac{\omega^n}{n!} = \frac{i^n}{\langle w \rangle^{2(n+1)}} dw^{\wedge} = \frac{2^n}{\langle w \rangle^{2(n+1)}} dw
\]
where $dw = dx_1\wedge dy_1\wedge\cdots\wedge dx_n\wedge dy_n$. We see $\mathcal{H}_k$ as a subspace of
\[
L^2\left(\mathbb{C},(\mathbb{C}^n)^{\otimes N},\frac{2^n}{\langle w \rangle^{2(n+1)}} dw\right).
\]

In the next computations we will need the following lemma
\begin{lemma}
\label{lcalcul}
Let $n\in\mathbb{N}^*$, $a\in\mathbb{N}^n$ and $b\in\mathbb{N}$ such that $|a|\le b$ then
\[
\int_{[0,+\infty[^n}\prod\limits_{1\le j\le n}r_j^{2a_j+1}\frac{1}{(1+r^2)^{b+n+1}}dr = \frac{1}{2^n(b+n)(b+n-1)\cdots(b+1)}\binom{b}{a}^{-1}
\]
where $r^2 = \sum\limits_{1\le j\le n}r_j^2$, and
\[
\binom{b}{a} = \frac{b!}{a_1!\cdots a_n!(b-|a|)!}.
\]
\end{lemma}

\begin{proof}
We integrate by parts multiple times, first on the $r_1$ variable
\begin{align*}
	& \int_{[0,+\infty[^n}\prod\limits_{1\le j\le n}r_j^{2a_j+1}\frac{1}{(1+r^2)^{b+n+1}}dr\\
	= & \frac{2a_1}{2(b+n)} \int_{[0,+\infty[^n}\prod\limits_{2\le j\le n}r_j^{2a_j+1}\frac{r_1^{2a_1-1}}{(1+r^2)^{b+n}}dr\\
	= & \cdots\\
	= & \frac{a_1(a_1-1)\cdots 2\cdot 1}{(b+n)b\cdots (b+n+1-a_1)} \int_{[0,+\infty[^n}\prod\limits_{2\le j\le n}r_j^{2a_j+1}\frac{r_1}{(1+r^2)^{b+n+1-a_1}}dr\\
	= & \frac{a_1!}{2(b+n)b\cdots (b+n-a_1)} \int_{[0,+\infty[^{n-1}}\prod\limits_{2\le j\le n}r_j^{2a_j+1}\frac{1}{(1+r_2^2+\cdots r_n^2)^{b+n-a_1}}dr.
\end{align*}
By iterating this process on each variable we get
\begin{align*}
	& \int_{[0,+\infty[^n}\prod\limits_{1\le j\le n}r_j^{2a_j+1}\frac{1}{(1+r^2)^{b+n+1}}dr\\
	= & \frac{a_1!}{2(b+n)b\cdots (b+n-a_1)} \int_{[0,+\infty[^{n-1}} \frac{ \prod\limits_{2\le j\le n}r_j^{2a_j+1} }{(1+r_2^2+\cdots r_n^2)^{b+n-a_1}}dr\\
	= & \frac{a_1!a_2!}{2^2(b+n)b\cdots (b+n-1-a_1-a_2)} \int_{[0,+\infty[^{n-2}} \frac{ \prod\limits_{3\le j\le n}r_j^{2a_j+1} }{(1+r_3^2+\cdots r_n^2)^{b+n-a_1}}dr\\
	= & \cdots\\
	= & \frac{a_1!a_2!\cdots a_n!}{2^n(b+n)b\cdots (b+n+1-n-a_1-a_2-\cdots-a_n)}\\
	= & \frac{a!(b-|a|)!}{2^n(b+n)!}
\end{align*}
\end{proof}

\begin{lemma}
Let $a\in\Nm^n$ such that $|a|\le N$, we write
\[
e_{N,a}(z) = \sqrt{\frac{(N+n)\cdots(N+1)}{(2\pi)^n}\binom{N}{a}}\frac{w^a\tilde{w}^{\otimes N}}{\langle w \rangle^{2N}}.
\]
Then $(e_{N,a})_{|a|\le N}$ is an orthonormal basis of $\mathcal{H}_N$. We will write $e_a$ instead to lighten the notation, and we note
\[
\Lambda_a = \sqrt{\frac{(N+n)\cdots(N+1)}{(2\pi)^n}\binom{N}{a}}.
\]
\end{lemma}

\begin{proof}
In view of the definition, this family is a basis of $\mathcal{H}_N$. If $a,b$ are multi-indexes such that $|a|,|b|\le N$
\[
\Lambda_a^{-1}\Lambda_b^{-1}\langle e_a,e_b\rangle = \int_{\mathbb{C}^n}\frac{\overline{w}^aw^b}{(1+|w|^2)^{N+n+1}}2^ndw
\]
then we use circular coordinates, that is $w_j = r_je^{i\theta_j}$ with $r_j\in [0,+\infty[$ and $\theta_j\in[0,2\pi]$.
\[
\Lambda_a^{-1}\Lambda_b^{-1}\langle e_a,e_b\rangle = 2^n\int_{[0,2\pi]^n}\int_{[0,+\infty[^n} \frac{1}{(1+r^2)^{N+n+1}} \prod\limits_{1\le j\le n}r_j^{a_j+b_j+1}e^{i\theta(b_j-a_j)}drd\theta.
\]
Hence, if $a\neq b$ at least one of the integrals in $\theta_j$ vanishes, thus the family is orthogonal. Now if $a=b$, by applying Lemma \ref{lcalcul}
\begin{align*}
\Lambda_a^{-2}\langle e_a,e_a\rangle & = (4\pi)^n \int_{[0,+\infty[^n} \frac{1}{(1+r^2)^{N+n+1}} \prod\limits_{1\le j\le n}r_j^{2a_j+1}dr\\
	& = \Lambda_a^2 \frac{(2\pi)^n}{(N+n)(N+n-1)\cdots(N+1)} \binom{N}{a}^{-1}\\
	& = 1.
\end{align*}
The family is orthonormal.
\end{proof}

\begin{corollary}
\label{corexproj}
$\Pi_N$ being the orthogonal projector on $\mathcal{H}_N$, we can write it with the orthonormal basis. Let $f\in L^2\left(\mathbb{C},(\mathbb{C}^n)^{\otimes N},\frac{\omega^n}{n!}\right)$ then
\[
\Pi_N(f) = \sum\limits_{|a|\le N} \left(\int_{\mathbb{C}^n} \frac{f(w)\cdot\overline{e_a}(w)}{(1+|w|^2)^{n+1}}2^ndw\right) e_a
\]
\end{corollary}

We now consider a specific operator which will saturate the estimate of Theorem \ref{th_mine}, we define the function $H\in C^0(\mathbb{CP}^n)$ such that
\[
H([z_0,z_1,\cdots,z_n]) = \frac{|z_1|^2}{|z_0|^2+|z_1|^2+\cdots+|z_n|^2}
\]
or in local coordinates
\[
H(w) = \frac{|w_1|^2}{1+|w|^2}.
\]

\begin{lemma}
Let $k\in\mathbb{N}$, then for $a\in\mathbb{N}^n$ such that $|a|\le k$
\[
T_N(H)(e_a) = \frac{a_1+1}{N+n+1}e_a
\]
\end{lemma}

\begin{proof}
According to corollary \ref{corexproj} and by definition of $T_N$
\[
T_N(H)(e_b) = \Pi_N(He_b) = \sum\limits_{0\le |a|\le N} \left(\int_{\mathbb{C}^n} H(w)\frac{e_b\cdot\overline{e_a}(w)}{\langle w \rangle^4}dw \right)e_a
\]
We compute for $|a|,|b|\le N$
\begin{align*}
		& \Lambda_a^{-1} \Lambda_b^{-1} \int_{\mathbb{C}^n} H(w)\frac{e_b\cdot\overline{e_a}(w)}{\langle w \rangle^4}2^ndw = 2^n\int_{\mathbb{C}^n} \frac{|w_1|^2w^b\overline{w}^a}{(1+|w|^2)^{N+n+2}}dw\\
	= & 2^n\int_{[0,2\pi]^n}\int_{[0,+\infty[^n} \frac{r_1^{a_1+b_1+3}}{(1+r^2)^{N+n+2}}\prod\limits_{2\le j\le n}r_j^{a_j+b_j+1} e^{i\theta_j(b_j-a_j)}drd\theta_1\cdots d\theta_n
\end{align*}
by taking polar coordinates on each variable $w_j$. If there is at least one $1\le j\le n$ such that $a_j\neq b_j$ then the integral in $\theta_j$ vanishes. Suppose now that $a=b$, we apply Lemma \ref{lcalcul}
\begin{align*}
		& \Lambda_a^{-2} \int_{\mathbb{C}^n} H(w)\frac{e_a\cdot\overline{e_a}(w)}{\langle w \rangle^4}2^ndw\\
	= & (4\pi)^n \int_{[0,+\infty[} \frac{r_1^{2a+3}}{(1+r^2)^{N+n+2}}\prod\limits_{2\le j\le n}r_j^{a_j+b_j+1}dr\\
	= & (2\pi)^n \frac{(a_1+1)a_1!(N+1-|a|-1)!}{(N+n+1)(N+n)!}\\
	= & \frac{a_1+1}{N+n+1} \Lambda_a^{-2}.
\end{align*}
Hence
\[
T_N(H)(e_a) = \frac{a_1+1}{N+n+1}e_a.
\]
\end{proof}

We are going to prove that the estimate of Theorem \ref{th_mine} is an asymptotic equivalence for this operator, hence the sharpness of the theorem. We take $\frac{1}{2}$ for the regular value of $H$, and we fix the sequence $(a_N)_{N\in\mathbb{N}}$ in $\mathbb{N}^n$, where $a_{N,1}$ is the integer part of $\frac{N}{2}$ and $a_{N,j}=0$ for $2\le j\le n$. That way, the eigenvalues $\frac{a_{N,1}+1}{N+n+1}$ converge to $\frac{1}{2}$.

\begin{theorem}
\label{th_optim_proj}
Let $p\in[2,+\infty]$ then there exists $C>0$ uniformly bounded with respect to $p$ such that
\[
\|e_{a_N}\|_{L^p(\mathbb{CP}^n)} \underset{N\to\infty}{\sim} CN^{\left(n-\frac{1}{2}\right)\left(\frac{1}{2}-\frac{1}{p}\right)}.
\]
\end{theorem}

\begin{proof}
We begin with $p=+\infty$. First, we fix $N\in\Nm$ and a general multi-index $a$ such that $|a|<N$. The function $|e_a|$ depends only on the variables $r_j = |w_j|$, we look for the critical values in these coordinates.
\[
\frac{\partial |e_a|}{\partial r_j} = \Lambda_a \left(\prod\limits_{i\neq j}r_i^{a_i}\right) r_j^{a_j-1}(1+r^2)^{-\frac{N}{2}-1}(a_j(1+r^2)-Nr_j^2).
\]
This expression vanishes if the $r_j$ are zero or if the last parenthesis vanishes, in which case
\[
\left(N I_n - A\right)
\begin{pmatrix}
r_1^2\\
\vdots\\
r_n^2
\end{pmatrix}
=
\begin{pmatrix}
a_1\\
\vdots\\
a_n
\end{pmatrix}
\]
where we write $A$ the matrix
\[
\begin{pmatrix}
a_1 & a_1 & \cdots & a_1\\
\vdots & \vdots & \vdots & \vdots \\
a_n & a_n & \cdots & a_n
\end{pmatrix}
.
\]
As long as $|a|<N$, $NI_N-A$ is invertible with inverse
\[
\frac{1}{N}I_N+\frac{1}{N(N-|a|)}A
\]
because $A^2 = |a|A$. Hence, we get the critical point $r$ such that $r_j = \sqrt{\frac{a_j}{N-|a|}}$, on which the Hessian of $|e_a|$ is equal to
\[
2\frac{(N-|a|)^{\frac{N-|a|}{2}+1}}{N^{\frac{N}{2}+1}}\prod\limits_{1\le j\le n} a_j^{\frac{a_j}{2}}\left((\sqrt{a_ia_j})_{1\le i,j\le n} -NI_n\right)
\]
which is definite negative. We can then check that $|e_a|$ admits a global maximum on that point, thus
\[
\|e_a\|_{L^{\infty}(\mathbb{PC}^n)} = \Lambda_a \frac{(N-|a|)^{\frac{N-|a|}{2}}}{N^{\frac{N}{2}}}\prod\limits_{1\le j\le n} a_j^{\frac{a_j}{2}}
\]
with convention $a_j^{\frac{a_j}{2}} = 1$ if $a_j=0$.

We now apply this result to the multi-indexes $a_N$, and we look for an asymptotic  equivalent when $N\rightarrow +\infty$. If $N$ is even $a_{N,1} = \frac{N}{2}$ and if $N$ is odd $a_{N,1} = \frac{N-1}{2}$, but either way

\begin{align*}
& \binom{N}{a_N} = \binom{N}{a_{N,1}} \underset{N\to\infty}{\sim} \sqrt{\frac{2}{\pi}}N^{-\frac{1}{2}}2^N,\\
& a_{N,1}^{\frac{a_{N,1}}{2}}N^{-\frac{N}{2}}(N-|a_N|)^{\frac{N-|a_N|}{2}} \underset{N\to\infty}{\sim} 2^{-\frac{N}{2}}.
\end{align*}
Thus

\[
\|e_{a_{N,1}}\|_{L^{\infty}(\mathbb{CP}^n)} \underset{N\to\infty}{\sim} \left(2^{n-\frac{1}{2}}\pi^{n+\frac{1}{2}}\right)^{-\frac{1}{2}} N^{\left(n-\frac{1}{2}\right)\frac{1}{2}}.
\]

We now take $p<+\infty$, in polar coordinates we get
\begin{align*}
\Lambda_{a_N}^p \|e_{a_N}\|^p_{L^p} & = \int_{\mathbb{C}^n}\frac{|w^{a_N}|^p}{(1+|w|^2)^{\frac{p}{2}N+n+1}}2^ndw\\
	& = (4\pi)^n \int_{[0,+\infty]^n} \frac{r_1^{pa_{N,1}+1}}{(1+r^2)^{\frac{p}{2}N+n+1}}\prod\limits_{2\le j\le n}r_jdr.
\end{align*}
We integrate by parts multiple times as in the proof of Lemma \ref{lcalcul}, but here only on the variables $r_j$ for $j\ge 2$
\[
\Lambda_{a_N}^p \|e_{a_N}\|^p_{L^p} = \frac{2(2\pi)^n}{(\frac{p}{2}N+n)\cdots (\frac{p}{2}N+2)}\int_0^{+\infty}\frac{r_1^{pa_{N,1}+1}}{(1+r^2)^{\frac{p}{2}N+2}}dr_1.
\]
Although, we cannot use the same method on $r_1$, as $p$ is not an integer in general. Instead, we use the change of variables $r_1\mapsto t=\frac{1}{1+r_1^2}$
\begin{align*}
\int_0^{\infty} \frac{r_1^{pa_{N,1}+1}}{(1+r_1^2)^{\frac{p}{2}N+2}}dr & = \frac{1}{2}\int_0^1 t^{\frac{p}{2}N+2}(1-t)^{\frac{pa_{N,1}+1}{2}}t^{-\frac{pa_{N,1}+1}{2}}t^{-\frac{3}{2}}(1-t)^{-\frac{1}{2}}dt\\
	& = \frac{1}{2} \int_0^1 t^{\frac{p}{2}(N-a_{N,1})} (1-t)^{\frac{p}{2}a_{N,1}} dt\\
	& = \frac{1}{2} B\left(\frac{p}{2}(N-a_{N,1})+1,\frac{p}{2}a_{N,1}+1\right)
\end{align*}
with $B$ the beta function. We know that (Chapter $3$ of \cite{arti64})
\[
B(a,b) = \frac{\Gamma(a)\Gamma(b)}{\Gamma(a+b)}
\]
and using the Stirling formula for $a_{N,1} = \frac{N}{2}$ and $a_{N,1} = \frac{N-1}{2}$, we find in both cases
\begin{align*}
B\left(\frac{p}{2}(N-a_{N,1})+1,\frac{p}{2}a_{N,1}+1\right) & = \frac{\Gamma\left(\frac{p}{2}(N-a_{N,1})+1\right)\Gamma\left(\frac{p}{2}a_{N,1}+1\right)}{\Gamma\left(\frac{p}{2}N+2\right)}\\
	& \underset{N\to\infty}{\sim} \sqrt{\frac{\pi}{p}}2^{-\frac{pN}{2}}N^{-\frac{1}{2}}.
\end{align*}
Combining the equivalences, we get
\begin{align*}
\|e_{a_N}\|^p_{L^p} & \underset{N\to\infty}{\sim} \Lambda_{a_N}^{-p}\frac{(2\pi)^n}{(\frac{p}{2}N+n)\cdots (\frac{p}{2}N+2)} \sqrt{\frac{\pi}{p}}2^{-\frac{pN}{2}}N^{-\frac{1}{2}}\\
		& \sim \left(\frac{(N+n)\cdots (N+1)}{(2\pi)^n}\sqrt{\frac{2}{\pi}}N^{-\frac{1}{2}}2^N\right)^{\frac{p}{2}} (2\pi)^n \left(\frac{Np}{2}\right)^{-n+1} \sqrt{\frac{\pi}{p}}2^{-\frac{pN}{2}}N^{-\frac{1}{2}}\\
		& \sim 2^{\left(n-\frac{1}{2}\right)\left(2-\frac{p}{2}\right)} p^{-\left(n-\frac{1}{2}\right)} \pi^{\left(n+\frac{1}{2}\right)\left(1-\frac{p}{2}\right)} N^{\left(n-\frac{1}{2}\right)\left(\frac{p}{2}-1\right)}
\end{align*}
then
\[
\|e_{a_N}\|_{L^p} \underset{N\to\infty}{\sim} \left(\frac{2}{p}\right)^{\frac{n-\frac{1}{2}}{p}} \frac{1}{\left(\pi^{\left(n+\frac{1}{2}\right)} 2^{\left(n-\frac{1}{2}\right)}\right)^{\frac{1}{2}-\frac{1}{p}}} N^{\left(n-\frac{1}{2}\right)\left(\frac{1}{2}-\frac{1}{p}\right)}
\]
\end{proof}

In the calculus of  $\|e_a\|_{L^{\infty}(\mathbb{PC}^n)}$ for a general $a$, we see that the only multi-indexes with $|a|<N$ for which $e_a$ saturates the estimate are the ones with $n-1$ bounded indexes and one index of order $N$.

\subsection{Adaptation to flat spaces}

Now that we proved the sharp concentration estimate for Toeplitz operators on compact manifolds, we are interested in adapting it for flat spaces. To do so, we detail the usual symbolic calculus of Toeplitz operators on $\Cm$, and apply it to link them with Toeplitz operators on the Torus. Moreover, the Toeplitz operators on $\Cm^n$ are also linked to pseudodifferential ones on $\Rm^n$ through a unitary transformation. Using this link and Theorem \ref{th_mine}, we obtain a result for pseudodifferential operators on $\Rm^n$. 

Actually, we consider here a different definition of the Bargmann space than in Section \ref{sec_Fock}, it is given by
\[
H_{\Phi} = L^2(\Cm^n) \cap \lacc e^{-N\Phi(z)}v \;/\; v\text{ is holomorphic} \racc,
\]
where $\Phi(z) = \frac{\Im(z)^2}{2}$, and it is also equipped with the $L^2$ scalar product. The Bergman projector is then given by
\begin{align*}
\Pi_{\Phi} : L^2(\Cm^n) & \rightarrow H_{\Phi}\\
u & \mapsto \lpar z \mapsto \lpar \frac{N}{2\pi}\rpar^n e^{-\frac{N\Im(z)^2}{2}} \int_{\Cm^n} e^{-\frac{N(z-\overline{w})^2}{4}} e^{\frac{N\Im(w)^2}{2}} u(w) dw \rpar.
\end{align*}
and for a function $f\in C(\Cm^n)$ we can define the associated Toeplitz operator $T_{\Phi}(f) : u \mapsto \Pi_{\Phi} (fu)$. This construction was detailed by Rouby, Sjöstrand and V\~u Ng\d{o}c \cite{roub20}.

We recall that an order function $m : \Cm^n \rightarrow ]0,+\infty [$ is such that there exists $C>0$ and $L\in\Rm$ with
\[
m(z) \le C \langle z-w\rangle^L m(w)
\]
for all $w,z\in \Cm^n$. The main result of this section is then

\begin{theorem}
\label{th_link_Toeplitz}
Let $m$ be an order function and $f\in S(m)$ be real-valued and such that $\lver f(z) \rver \xrightarrow[|z|\rightarrow +\infty]{} +\infty$. Let $V_N \in H_{\Phi}(\Cm^n)$ be $L^2$-normalised and such that
\[
T_N(f)V_N = E_N V_N + O(N^{-\infty}),
\]
and the sequence $E_N$ converges to a regular value of the principal symbol of $f$. Then
\[
\lnor V_N \rnor_{L^p(\Cm^n)} = O\lpar N^{ \lpar n-\frac{1}{2} \rpar\lpar \frac{1}{2}-\frac{1}{p} \rpar } \rpar.
\]
\end{theorem}

This quantization is different from Section \ref{sec_Fock} because of the structure of the proof of Theorem \ref{th_link_Toeplitz}. Since Theorem \ref{th_mine} works only on a compact manifold, we will localise a quasimode on a square, and modify it to make it periodic, which can be seen as a section of the torus. The point is, this quantization is more suitable to link the Toeplitz quantization on $\Cm^n$ of periodic functions and the quantization of the torus (See for example \cite{roub17}).

Before proving Theorem \ref{th_link_Toeplitz}, we need a symbolic calculus for this quantization. The following constructions and results are merely adaptations of the Chapter 4 of \cite{zwor12} to Berezin-Toeplitz operators.

\begin{definition}
For an order function $m : \Cm^n \rightarrow ]0,+\infty [$, let us define the associated space of symbols
\[
S(m) = \lacc f\in C^{\infty}(\mathbb{C}^n) \; /\; \forall \alpha \in \Nm^{2n}, \; \exists C_{\alpha}, \; |\partial^{\alpha_1} \overline{\partial}^{\alpha_2} f| \le C_{\alpha} m \racc
\]
and if $(f_j)_{j\in\Nm} \in S(m)^{\Nm}$, we say that $f\in S(m)$ is asymptotic to $\sum N^{-j} f_j$ if, for all $L\in\Nm$, $f-\sum\limits_{0\le j\le L-1}N^{-j} f_j = O_{S(m)}(N^{-L})$. This last equality means that for all $\alpha\in\Nm^2$ there exists $C_{\alpha} >0$ with
\[
\lver \partial^{\alpha_1} \overline{\partial}^{\alpha_2} \lpar 	f - \sum\limits_{0\le j\le L-1} N^{-j} f_j \rpar \rver \le \frac{C_{\alpha} m}{N^L}.
\]
By Borel's theorem, for all $(f_j)_{j\in\Nm} \in S(m)^{\Nm}$ there exists such symbol $f$, and we write $f\sim \sum N^{-j} f_j$. For $f\in S(m)$ we still define the associated Toeplitz operator by
\[
T_N(f):v \mapsto \Pi_{\Phi}(fv)
\]
with domain
\[
D(T_N(f)) = \lacc v\in H_{\Phi} / mv\in L^2\racc.
\]
\end{definition}

The following result gives the principal symbol of the composition of two Toeplitz operators. This result can be proved directly by calculus, or using the link with pseudodifferential operators and the corresponding Theorem in this context. We don't give the direct proof here as we will explicit the link with pseudodifferential operators in Proposition \ref{prop_link_pseudo_toeplitz}.

\begin{proposition}
In $\Cm^n$, let $m_1,m_2$ be two order functions, then for all $f,g \in S(m_1) \times S(m_2)$ there exists a symbol $f\# g\in S(m_1m_2)$ such that
\[
T_N(f) T_N(g) = T_N(f\# g)
\]
Furthermore, we have the expansion
\[
f\#g \sim fg + N^{-1} r
\]
where $r\in S(m_1m_2)$.
\end{proposition}

\begin{theorem}[Inverse for elliptic operators]
\label{th_elliptic_inverse}
Let $f \in S(m)$ be an elliptic symbol, that is $|f(z)| \ge c m(z)$ for all $z\in\Cm^n$ with $c>0$. Then there exists $g\in S(m^{-1})$ such that
\[
T_N(g) T_N(f) = id_{H_{\Phi}} + R
\]
where, for all $L\in\Nm$ there exists $C_L>0$ such that for all $v\in H_{\Phi}$
\begin{equation}
\label{eq_elliptic_inverse}
\lnor R v \rnor_{L^2} \le N^{-L} C_L \lnor v \rnor_{L^2}.
\end{equation}
\end{theorem}

\begin{proof}
We prove by induction on $k\in\Nm$  that there exists symbols $g_k \in S(m^{-1})$ and $r_k \in S(1)$ such that
\[
T_N(g_k) T_N(f) = id_{\fbarg} + N^{-k} T_N(r_k).
\]
First, we write $g_0 = f^{-1}$ which is well-defined in $S(m^{-1})$ by assumption on $f$, and since $g_0 \# f = 1 + \hbar r_0$ we have
\[
T_N(g_0) T_N(f) = id_{H_{\Phi}} + N^{-1} R_0.
\]
Now we suppose the assumption to be true for $k\in\Nm$, so writing
\[
g_{k+1} = \lpar 1-N^{-1} r_k \rpar \# g_k \in S(m^{-1})
\]
we get 
\[
T_N(g_{k+1}) T_N(f) = id_{H_{\Phi}} + N^{-(k+1)} T_N(r_{k+1}).
\]
By construction, for all $k\in\Nm$ $g_{k+1} = g_k + N^{-1} q_k$ where $q_k \in S(1)$, thus there exists $g\in S(1)$ such that $g \sim \sum g_k$, that is
\[
T_N(g) T_N(f) = id_{H_{\Phi}} + T_N(r)
\]
where $r = O(N^{-\infty})$ in $S(1)$. In particular, for all $L\in\Nm$

\begin{align*}
\lnor T_N(r) v \rnor 
	& \le N^{-L} \lnor \Pi_N \lpar r v \rpar \rnor\\
	& \le N^{-L} \lnor r v \rnor\\
	& \le N^{-L} \lnor r \rnor_{L^{\infty}} \lnor v \rnor_{L^2}
\end{align*}
hence equation \eqref{eq_elliptic_inverse}.
\end{proof}

\begin{corollary}
\label{cor_elliptic_inverse}
Let $f\in S(m)$ be such that $|f(z)| \rightarrow +\infty$ when $|z| \rightarrow +\infty$, and $V_N\in H_{\Phi} $ with $T_N(f)V_N=E V_N + O(N^{-\infty})$, that is for all $L\in\Nm$ there exists $C_L>0$ such that
\[
\lnor T_N(f)V_N - EV_N \rnor_{L^2} \le C_L N^{-L}.
\]
Then there exists $V'_N \in H_{\Phi}$ with compact support such that
\[
V_N = V'_N + O(N^{-\infty}).
\]
and it is still a quasimode,
\[
T_N(f) V'_N = O(N^{-\infty}).
\]
\end{corollary}

\begin{proof}
We suppose that $\lambda = 0$ since it is the same result by taking $f+\lambda$ instead of $f$. There exists $\chi \in S(1)$ with compact support such that $f+\chi$ is elliptic in $S(m')$ for an order function $m'\ge m$, thus there exists $q\in S((m')^{-1})$ such that
\[
T_N(q)T_N(f+\chi) = id_{H_{\Phi}} + R
\]
with $R$ satisfying \eqref{eq_elliptic_inverse}, because of Theorem \ref{th_elliptic_inverse}. We consider the function of $H_{\Phi}$
\[
v' = T_N(q)T_N(\chi)v = v - T_N(q)T_N(f)v + Rv = v + O(N^{-\infty})
\]
which has compact support and satisfies $v-v' = O(N^{-\infty})$. Furthermore, since $Rv = O(N^{-\infty})$,
\[
T_N(f) Rv = O(N^{-\infty})
\]
and $fq\in S(1)$ so
\[
\lnor T_N(f) T_N(q) T_N(f) v \rnor_{L^2} \le C \lnor T_N(f) v \rnor_{L^2} = O(N^{-\infty}).
\]
In particular
\[
T_N(f) v' = T_N(f)v - T_N(f) T_N(q) T_N(f) v +T_N(f) Rv = O(N^{-\infty}).
\]

\end{proof}

With these tools we can now prove Theorem \ref{th_link_Toeplitz}. We remind the reader that we consider here a real-valued symbol $f\in S(m)$ such that $\lver f(z) \rver \xrightarrow[|z|\rightarrow +\infty]{} +\infty$ and $V_N \in H_{\Phi}(\Cm^n)$ a $L^2$-normalised quasimode.

\begin{proof}[Proof of Theorem \ref{th_link_Toeplitz}]
According to Corollary \ref{cor_elliptic_inverse}, there exists $V'_N$ with compact support which is also a quasimode of $T_N(f)$ and such that

\[
V'_N = V_N + O(N^{-\infty}).
\]
In particular, we can choose it so that it also has unit $L^2$-norm, and we define $\Sigma$ a square in $\Cm^n$ such that the support of $V'_N$ lies in it.

We want to see $V'_N$ as a section of the torus $\Tm^n = \lpar [0,2\pi]\times [0,1] \rpar^n$ in order to apply Theorem \ref{th_mine}. To do so, we write $\kappa : \Tm^n \rightarrow \Sigma$ the affine change of variable, and we will work on $W_N = V'_N \circ \kappa$. This new function is a quasimode for the operator with symbol $\tilde{f} = f \circ \kappa$, and since $W_N$ has compact support in $\Tm^n$, we can modify $\tilde{f}$ near its boundaries such that it is periodic with respect to $\Tm^n$ and that $W_N$ is still a quasimode for the associated operator. Hence, we can see $W_N$ as a section on the torus $\Tm$ and $T_N(\tilde{f})$ as a Berezin-Toeplitz operator on this manifold. By hypothesis on $f$ and the definition of $\tilde{f}$, all the hypothesis of Theorem \ref{th_mine} are fulfilled, thus we get that

\[
\lnor W_N \rnor_{L^p(\Tm)} = O\left(N^{\left(n-\frac{1}{2}\right)\left(\frac{1}{2}-\frac{1}{p}\right)}\right)
\]
and the change of variable $\kappa$ only adds a constant so

\[
\lnor V'_N \rnor_{L^p(\Sigma)} = O\left(N^{\left(n-\frac{1}{2}\right)\left(\frac{1}{2}-\frac{1}{p}\right)}\right).
\]
Since $V'_N$ has support in $\Sigma$ we can replace $L^p(\Sigma)$ by $L^p(\Cm^n)$ in the last equation. Now, in order to bound $V_N$ we notice that for any $U_N \in H_{\Phi}$

\[
\lver \Pi_{\Phi} U_N(z) \rver \le \lver e^{-\frac{N|\bullet|^2}{4}} \ast U_N \rver(z)
\]
and by Young's inequality with $1+\frac{1}{p}=\frac{1}{2}+\frac{1}{r}$, we get that

\[
\lnor U_N \rnor_{L^p} =
\lnor \Pi_{\Phi} U_N  \rnor_{L^p} \le
\lnor e^{-\frac{N|\bullet|^2}{4}} \rnor_{L^r} \lnor U_N \rnor_{L^2} \le
C N^{-\frac{n}{r}} \lnor U_N \rnor_{L^2}.
\]
Then, using that $V_N-V'_N \in H_{\Phi}(\Cm^n)$ and

\[
\lnor V_N - V'_N \rnor_{L^2(\Cm^n)} = O \lpar N^{-\infty} \rpar,
\]
we get that $\lnor V_N - V'_N \rnor_{L^p(\Cm^n)} = O \lpar N^{-\infty} \rpar$ for all $p\in[2,+\infty]$ and thus the bound on $V_N$.
\end{proof}

We considered the space $H_{\Phi}$ here and $\fbarg$ in Section \ref{sec_Fock}, in fact there exists a whole class of "Bargmann-type" spaces on which we can build Berezin-Toeplitz operators. It can be done for any quadratic function $\Phi$ such that $\partial\overline{\partial}\Phi>0$, but all these spaces are unitary equivalent anyway (see \cite{roub20}).

We see with these results that Toeplitz operators have a similar behaviour than pseudodifferential ones. In fact, the next result shows a unitary equivalence between these two.

\begin{proposition}[\cite{roub20} Chapter 2]
\label{prop_link_pseudo_toeplitz}
Let $m$ be an order function and $a \in C^{\infty}(\Rm^{2n})$ be a real symbol of order $m$, then there exists a symbol $f\in S(m)$ such that
\[
T_N(f) = \fbi[N] Op^w(a_{\hbar}) \fbi[N]^* + O(N^{-\infty})
\]
where
\begin{align*}
	\fbi[N] : L^2(\Rm^n) & \rightarrow H_{\Phi}(\Cm^n)\\
				u & \mapsto 2^{-\frac{n}{2}} \lpar \frac{N}{\pi} \rpar^{\frac{3n}{4}} e^{-\frac{N\Im(z)^2}{2}} \int_{\Rm^n} e^{-\frac{N}{2}(z-x)^2}u(x)dx\\
	\fbi[N]^* : v & \mapsto \lpar \frac{N}{\pi} \rpar^{\frac{3n}{4}} \int_{\Cm^n} e^{-\frac{N}{2}(\overline{z}-x)^2} e^{-\frac{N\Im(z)^2}{2}} v(z)dz
\end{align*}
is unitary and satisfy $\fbi[N]^* \fbi[N] = id_{L^2(\Rm^n)}$. This is the FBI transform for the space $H_{\Phi}$. Furthermore, we have the asymptotic
\[
f(z) \sim e^{-\frac{\hbar}{4}\Delta} a \lpar \Re(z),-\Im(z) \rpar.
\]
\end{proposition}

We now want to highlight the fact that we can combine Theorem \ref{th_link_Toeplitz} and Proposition \ref{prop_link_pseudo_toeplitz} to get a $L^p$ bound for the FBI transform of pseudodifferential operators' quasimodes. Although the bound is not on the function itself, we find interesting that it is sharply bound once a $L^2$-unitary application is used.

\begin{theorem}
\label{th_link_pseudo}
Let $m$ be an order function and $a \in C^{\infty}(\Rm^{2n})$ be a real symbol of order $m$, that is for all $\beta_1,\beta_2 \in \Nm^{2n}$ there exists $C_{\beta} >0$ such that
\[
\lver \partial_x^{\beta_1} \partial_{\xi}^{\beta_2} a \rver \le C_{\beta} m.
\]
Suppose furthermore that $|a(x,\xi)| \xrightarrow[|x|,|\xi| \rightarrow +\infty]{} +\infty$. If $u_N$ is a quasimode of $Op^w(a)$ with unit $L^2$ norm such that the associated eigenvalues converge to a regular value of the principal symbol of $a$, then for $2\le p \le\infty$
\[
\lnor \fbi[N] u_N \rnor_{L^p(\Cm^n)} = O \lpar N^{\lpar n-\frac{1}{2} \rpar\lpar 1-\frac{2}{p} \rpar } \rpar.
\]
\end{theorem}

\begin{proof}
We write $V_N = \fbi[N] u_N$ which also has unit $L^2_{\Phi}$ norm since $\fbi[N]$ is unitary. Because of Proposition \ref{prop_link_pseudo_toeplitz}, we have $u_N = \fbi[N]^* \lpar \fbi[N] u_N \rpar = \fbi[N]^* V_N $ and
\begin{align*}
T_N(f) V_N
	& = \fbi[N] Op^w(a)(x,\hbar D) \fbi[N]^* V_N + O_{L^2(\Rm^n)}(N^{-\infty})\\
	& = \fbi[N] Op^w(a)(x,\hbar D) u_N + O_{L^2(\Rm^n)}(N^{-\infty})\\
	& = E_N \fbi[N] u_N + \fbi[N]O_{L^2(\Rm^n)}(N^{-\infty}) + O_{L^2(\Rm^n)}(N^{-\infty})\\
	& = E_N V_N + O_{L^2(\Cm^n)}(N^{-\infty})
\end{align*}
where $f(z) \sim e^{-\frac{1}{4N}\Delta} a \lpar \Re(z),-\Im(z) \rpar$. By hypothesis on $a$, the symbol $f$ satisfies the hypothesis of Theorem \ref{th_link_Toeplitz}, hence

\[
\lnor \fbi[N] u_N \rnor_{L^p(\Cm^n)} = O\lpar N^{\lpar n-\frac{1}{2} \rpar\lpar \frac{1}{2}-\frac{1}{p} \rpar} \rpar,
\]
for every $p\in[2,\infty]$.
\end{proof}

One could try to deduce a bound on the $L^p$ norms of $u_N$, though it would give a very suboptimal result, meanwhile a simple elliptic argument gives that for any $\epsilon>0$ there exists $C>0$ such that for any $u_N$ as in Theorem \ref{th_link_pseudo} we have
\[
\lnor u_N \rnor_{L^p(\Rm^n)} \le C N^{\lpar \frac{n}{2}+\epsilon \rpar\lpar \frac{1}{2}-\frac{1}{p} \rpar}.
\]

The purpose of this article was to prove concentration estimates for Toeplitz operators in a large framework. Although, many results for pseudodifferential operators consider more specific cases, for example the joint eigenfunctions of completely integrable systems, operators with symmetries or different dynamical assumptions, as Anosov flows. There are no equivalents for Toeplitz operators, and it would be interesting to see how they adapt with the methods used here. Also, one could wonder if Theorem \ref{th_link_pseudo} is sharp.

\appendix

\section{Quantization on the Bargmann space}
\label{sec_constr}

For the sake of completeness, we give here the detailed Toeplitz quantization on $\Cm$. In this section we will use the Fourier transform, we write it for $f\in C_c^{\infty}(\mathbb{R}^n)$ and $\xi \in \mathbb{R}^n$,
\[
\mathcal{F}(f)(\xi) = \frac{1}{(2\pi h)^{\frac{n}{2}}} \int e^{-i\frac{\xi \cdot x}{h}} f(x) dx.
\]

\begin{definition}
We define the FBI transform as
\begin{align*}
\fbi[N]: L^2(\mathbb{R}^n)  & \rightarrow L^2(\mathbb{C}^n)\\
					f & \mapsto \left(\frac{N}{\pi}\right)^{\frac{n}{2}} 2^{\frac{n}{4}} \int_{\mathbb{R}^n} e^{2\sqrt{\pi N}x\cdot z} e^{-\pi x^2} e^{-\frac{Nz^2}{2}} f(x)dx \; e^{-\frac{N|z|^2}{2}}.
\end{align*}
\end{definition}

Using $\fbi[N]$, we can construct $\fbarg$ from $L^2(\mathbb{R}^n)$.

\begin{lemma}
$\fbi[N]$ is an isometry from $L^2(\mathbb{R}^n)$ to $\fbarg$.
\end{lemma}

\begin{proof}
For $f\in \mathcal{S}(\mathbb{R}^n)$ writing $z=r+is$,

\begin{align*}
\fbi[N] f(z)
		& = \left(\frac{N}{\pi}\right)^{\frac{n}{2}} 2^{\frac{n}{4}} \int_{\mathbb{R}^n}f(x) e^{2\sqrt{\pi N}x\cdot (r+is)} e^{-\pi x^2} e^{-\frac{N(r+is)^2}{2}}dx \; e^{-\frac{N|r+is|^2}{2}},\\
		& = \left(\frac{N}{\pi}\right)^{\frac{n}{2}} 2^{\frac{n}{4}} \int_{\mathbb{R}^n}f(x) e^{2\sqrt{N\pi}x\cdot r} e^{2i\sqrt{N\pi}x\cdot s} e^{-\pi x^2} e^{-Nr^2} e^{-iNr\cdot s}dx,\\
		& = \left(\frac{N}{\pi}\right)^{\frac{n}{2}} 2^{\frac{n}{4}} \int_{\mathbb{R}^n}f(x) e^{2i\sqrt{N\pi}x\cdot s} e^{-\pi (x-\sqrt{\frac{N}{\pi}}r)^2} dx \; e^{-iNr\cdot s},\\
		& = \left(\frac{N}{\pi}\right)^{\frac{n}{2}} 2^{\frac{n}{4}} (2\pi )^{\frac{n}{2}}\mathcal{F}^{-1}(fe^{-\pi(\cdot -\sqrt{\frac{N}{\pi}}r)^2})(2\sqrt{N\pi}s) \; e^{-iNr\cdot s}.\\
\end{align*}
Since $f$ is in the Schwartz space we can apply the Plancherel formula,

\begin{align*}
\|\fbi[N]f\|^2_{L^2(\mathbb{C}^n)} 
	& = \left(\frac{N}{\pi}\right)^n 2^{\frac{n}{2}} (2\pi )^n\int_{\mathbb{R}^{2n}}|\mathcal{F}^{-1}(fe^{-\pi(\cdot -\sqrt{\frac{N}{\pi}}r)^2})(2\sqrt{N\pi}s)|^2drds,\\
	& = \left(\frac{N}{\pi}\right)^{\frac{n}{2}} 2^{\frac{n}{2}} \int_{\mathbb{R}^{2n}} |f(s)|^2 e^{-2\pi(s -\sqrt{\frac{N}{\pi}}r)^2}drds,\\
	& = \|f\|^2_{L^2(\mathbb{R}^n)}.
\end{align*}
Using the density of the Schwartz space in $L^2$ we get that $\fbi[N]$ is an isometry, also
\[
e^{-\frac{N|z|^2}{2}}\fbi[N]f(z) = e^{-\frac{Nz^2}{2}} \left(\frac{N}{\pi}\right)^{\frac{n}{2}} 2^{\frac{n}{4}} \int_{\mathbb{R}^n}f(x) e^{2\sqrt{\pi N}x\cdot z} e^{-\pi x^2} dx
\]
so by holomorphy under the integral $e^{-\frac{N|z|^2}{2}}\fbi[N]f(z)$ is holomorphic, thus $\fbi[N]$ has values in $\fbarg$.

\end{proof}

We recall that $\fbarg$ has the Hilbert basis $\left(e_{\alpha} = \frac{N^{\frac{n+|\alpha|}{2}}e^{-\frac{N|z|^2}{2}}z^{\alpha}}{\pi^{\frac{n}{2}}\sqrt{\alpha !}}\right)_{\alpha \in \mathbb{N}^n}$, thanks to Proposition \ref{prop_base}. It allows us to explicit the reproducing kernel of $\fbarg$, that is  the functions $K:\mathbb{C}^n \times \mathbb{C}^n \rightarrow \mathbb{C}$ such that for all $g\in\fbarg$, and for all $z\in\mathbb{C}^n$

\[
g(z) = \langle K(\cdot ,z),g\rangle = \int_{\mathbb{C}^n} \overline{K(w,z)}g(w)dw.
\] 

\begin{corollary}[\cite{barg61} Chapter 1.c]
The function
\[
K: (w,z) \rightarrow \left(\frac{N}{\pi}\right)^n e^{-\frac{N|z|^2}{2}} e^{-\frac{N|w|^2}{2}} e^{N\overline{z}\cdot w} = \sum\limits_{\alpha\in\mathbb{N}^n} \overline{e_{\alpha}(z)}e_{\alpha}(w)
\]
is the reproducing kernel in $\fbarg$. 
\end{corollary}

\begin{proof}
$K$ is well-defined and the series converge uniformly on any compact because of the analyticity of the exponential. The reproducing property comes from Proposition \ref{prop_base}.
\end{proof}

\begin{example}
Here is some example of Toeplitz operators of simple functions.
\begin{itemize}
\item[i] If $f\in\mathcal{H}(\mathbb{C}^n)$ then $T_N(f)=fid_{\fbarg}$.
\item[ii] If $f\in L^{\infty}$, we can define in the same manner $T_N(f)$, and it will be a bounded operator on $\fbarg$.
\item[iii] If $f$ is real-valued then $T_N(f)$ is symmetric.
\item[iv] Let $j\in\{1,\cdots ,n\}$, we write
\begin{align*}
D_j:\fbarg & \rightarrow\fbarg\\
g & \mapsto \frac{e^{-\frac{N|z|^2}{2}}}{N}\partial_{z_j}\lpar e^{\frac{N|z|^2}{2}}g\rpar
\end{align*}
then for all $\alpha ,\beta\in\mathbb{N}^n$,
\[
T_N(z^{\alpha}\overline{z}^{\beta})= D^{\beta}\circ (z^{\alpha}id_{\fbarg})
\]
\end{itemize}
where $D^{\beta} = D_1^{\beta_1} \circ\cdots\circ D_n^{\beta_n}$.
\end{example}

\begin{proof}~
\begin{itemize}
\item[i] If $f$ is a holomorphic function then for all $u\in\fbarg$, $e^{N\frac{|z|^2}{2}}fu$ is one too, so $fu\in\fbarg$ then $T_N(f)(u)=\Pi_N(fu) = fu$.
\item[ii] For such $f$ and for all $u\in\fbarg$, $fu\in L^2$, so we can apply $\Pi_N$. Furthermore, by continuity of the projector and Young inequality,
\[
\|\Pi_N(fu)\|_{L^2(\mathbb{C}^n)} \le \|fu\|_{L^2(\mathbb{C}^n)} \le \|f\|_{L^{\infty}(\mathbb{C}^n)} \|u\|_{L^2(\mathbb{C}^n)}.
\]
Thus $T_N(f): u \mapsto \Pi_N(fu)$ is defined and bounded on $\fbarg$.
\item[iii] We deduce it from the fact that $\Pi_N$ if self-adjoint.
\end{itemize}
\end{proof}

\section{References}

\sloppy 
\printbibliography[heading=none]

\end{document}